\documentclass[hidelinks,onefignum,onetabnum]{siamart251216}

\newif\ifisarxiv
\isarxivtrue

\usepackage{lipsum}
\usepackage{amsfonts}
\usepackage{amssymb}
\usepackage{graphicx}
\usepackage{epstopdf}
\usepackage[normalem]{ulem}
\usepackage{subcaption}  
\usepackage{comment}
\usepackage{algorithm, algpseudocode, algorithmicx}

\ifpdf
  \DeclareGraphicsExtensions{.eps,.pdf,.png,.jpg}
\else
  \DeclareGraphicsExtensions{.eps}
\fi

\def\berr{{\mathrm{berr}}}
\def\algpsd{{MINBERR}}
\def\algne{{MINBERR-NE}}

\newsiamremark{remark}{Remark}
\newsiamremark{hypothesis}{Hypothesis}
\crefname{hypothesis}{Hypothesis}{Hypotheses}
\newsiamthm{claim}{Claim}
\newsiamremark{fact}{Fact}
\crefname{fact}{Fact}{Facts}

\def\myfunding{MD was supported in part by NSF CAREER Grant CCF-233865 and a Google ML and Systems Junior Faculty Award. This work was done in part while MD was visiting the Simons Institute for the Theory of Computing. 
YN was supported by the EPSRC grant EP/Y030990/1.
ER was partially supported by NSF DMS-2309685. }

\headers{Universal Convergence of Linear  Solvers}{M. Derezi\'nski, Y. Nakatsukasa, and E. Rebrova}

\title{Towards Universal Convergence\\ of Backward Error
in Linear System Solvers\thanks{\ifisarxiv\textbf{Funding:} \myfunding\else Submitted to the editors \today.\funding{\myfunding}\fi
 }}

\author{Micha{\L} Derezi\'nski\thanks{University of Michigan, Ann Arbor, MI
  (\email{derezin@umich.edu}).}
\and Yuji Nakatsukasa\thanks{University of Oxford, Oxford, UK  (\email{yuji.nakatsukasa@maths.ox.ac.uk}).}
\and Elizaveta Rebrova\thanks{Princeton University, Princeton, NJ (\email{elre@princeton.edu}).}}

\usepackage{amsopn}

\def\Kc{\mathcal{K}}

\newif\ifDRAFT
\DRAFTtrue
\ifDRAFT
\newcommand{\marrow}{\marginpar[\hfill$\longrightarrow$]{$\longleftarrow$}}
\newcommand{\niceremark}[3]
   {\textcolor{red}{\textsc{#1 #2:} \marrow\textsf{#3}}}
\newcommand{\michal}[2][says]{\niceremark{Michal}{#1}{#2}}
\else
\newcommand{\michal}[1]{}
\fi

\def\poly{{\mathrm{poly}}}

\def\Sigmab{\mathbf{\Sigma}}

\def\T{\mathbf{T}}

\def\q {\mathbf{q}}

\def\G{{\mathbf{G}}}

\def\W{\mathbf W}

\def\Pc{\mathcal{P}}

\def\Tc{\mathcal{T}}

\def\Q{\mathbf Q}

\newcommand{\Span}{\mathrm{span}}

\ifx\BlackBox\undefined
\newcommand{\BlackBox}{\rule{1.5ex}{1.5ex}}  
\fi
\DeclareMathOperator*{\argmin}{\mathop{\mathrm{argmin}}}

\def\x{\mathbf x}
\def\y{\mathbf y}

\def\z{\mathbf z}

\def\b{\mathbf b}
\def\w{\mathbf w}
\def\v{\mathbf v}

\def\e{\mathbf e}
\def\zero{\mathbf 0}

\def\u{\mathbf u}

\def\B{\mathbf B}
\def\A{\mathbf A}

\def\U{\mathbf U}

\def\V{\mathbf V}
\def\M{\mathbf M}

\def\I{\mathbf I}

\def\A{\mathbf A}

\def\R{\mathbb R}

\let\origtop\top
\renewcommand\top{{\scriptscriptstyle{\origtop}}} 

\definecolor{silver}{cmyk}{0,0,0,0.3}
\definecolor{yellow}{cmyk}{0,0,0.9,0.0}
\definecolor{reddishyellow}{cmyk}{0,0.22,1.0,0.0}
\definecolor{black}{cmyk}{0,0,0.0,1.0}
\definecolor{darkYellow}{cmyk}{0.2,0.4,1.0,0}
\definecolor{orange}{cmyk}{0.0,0.7,0.9,0}
\definecolor{darkSilver}{cmyk}{0,0,0,0.1}
\definecolor{grey}{cmyk}{0,0,0,0.5}
\definecolor{darkgreen}{cmyk}{1,0,1,0} 

\ifx\proof\undefined
\newenvironment{proof}{\par\noindent{\bf Proof\ }}{\hfill\BlackBox\\[2mm]}
\fi

\ifx\theorem\undefined
\newtheorem{theorem}{Theorem}
\fi

\ifx\example\undefined
\newtheorem{example}{Example}
\fi

\ifx\condition\undefined
\newtheorem{condition}{Condition}
\fi
\ifx\property\undefined
\newtheorem{property}{Property}
\fi

\ifx\lemma\undefined
\newtheorem{lemma}{Lemma}
\fi

\ifx\proposition\undefined
\newtheorem{proposition}{Proposition}
\fi

\ifx\remark\undefined
\newtheorem{remark}{Remark}
\fi

\ifx\corollary\undefined
\newtheorem{corollary}{Corollary}
\fi

\ifx\definition\undefined
\newtheorem{definition}{Definition}
\fi

\ifx\conjecture\undefined
\newtheorem{conjecture}{Conjecture}
\fi

\ifx\axiom\undefined
\newtheorem{axiom}{Axiom}
\fi

\ifx\claim\undefined
\newtheorem{claim}{Claim}
\fi

\ifx\assumption\undefined
\newtheorem{assumption}{Assumption}
\fi

\ifx\condition\undefined
\newtheorem{condition}{Condition}
\fi

\ifpdf
\hypersetup{
  pdftitle={An Example Article},
  pdfauthor={D. Doe, P. T. Frank, and J. E. Smith}
}
\fi

\begin{document}

\maketitle

\begin{abstract}
The quest for an algorithm that solves an $n\times n$ linear system
in $O(n^2)$ time complexity, or $O(n^2 \poly(1/\epsilon))$ 
when solving up to $\epsilon$ relative error, is a long-standing open problem in numerical linear algebra and theoretical computer science. 
There are two predominant paradigms for measuring relative error:
forward error (i.e., distance from the output to the optimum solution) and backward error (i.e., distance to the nearest problem solved by the output). In most prior studies, convergence of iterative linear system solvers is measured via various notions of forward error, and as a result, depends heavily on the conditioning of the input. Yet, the numerical analysis literature has long advocated for backward error as the more practically relevant notion of approximation. In this work, we show that --- surprisingly --- the classical and simple Richardson iteration incurs at most $1/k$ (relative) backward error after $k$ iterations on any positive semidefinite (PSD) linear system, irrespective of its condition number. This universal convergence rate implies an $O(n^2/\epsilon)$ complexity algorithm for solving a PSD linear system to $\epsilon$ backward error, and we establish
similar or better complexity when using a variety of Krylov solvers beyond Richardson. Then, by directly minimizing backward error over a Krylov subspace, we attain an even faster $O(1/k^2)$ universal rate, and we turn this into an efficient algorithm, MINBERR, with complexity $O(n^2/\sqrt\epsilon)$. Finally, we extend this approach via normal equations to solving general linear systems in $O(n^2\log(n)/\epsilon)$ time complexity. We report strong numerical performance of our algorithms on benchmark problems.
\end{abstract}

\begin{keywords}
Linear systems, Backward error, Krylov methods, Richardson iteration
\end{keywords}

\begin{MSCcodes}
	65F10, 65F35, 65Y20
\end{MSCcodes}

\section{Introduction}
Solving a system of linear equations $\A\x=\b$ given an $n\times n$ matrix $\A$ and vector $\b\in\mathbb{R}^n$ is perhaps the most important problem in computational mathematics. 
Direct methods, such as those based on Gaussian elimination with partial pivoting or QR factorization, remain the de facto standard for solving linear systems when $n$ is moderate, say $n<10^4$ in today's desktop computers and $n<10^6$ with supercomputers. 
However, the arithmetic complexity of these methods scales as $O(n^3)$, making them infeasible when $n$ is much larger. An alternative approach is to use fast Strassen-type matrix multiplication algorithms, which improve the asymptotic complexity of solving linear systems to $O(n^\omega)$, where currently $\omega\approx 2.371$ \cite{alman2025more}, but have remained largely theoretical due to extremely large constant factors.

All of the aforementioned approaches tend to produce highly accurate solutions (modulo numerical precision issues), but at a relatively high computational cost. Yet, in many applications of linear system solvers we are willing to trade off accuracy for computation by relying on approximate solutions that are computed within some $\epsilon$ error relative to the optimum. This has motivated the long-standing question:
\smallskip
\begin{center}
\textit{Can we compute an $\epsilon$ relative error solution of an $n\times n$ linear system in $O(\,n^2\,\poly(1/\epsilon)\,)$ time?}
\end{center}
\smallskip

In response to this question, many iterative methods have been developed, from the classical iteration scheme proposed by Richardson in 1911 \cite{richardson1911ix} to the modern Krylov subspace methods \cite{greenbaumbook}. These algorithms can quickly produce a rough approximation to a linear system by making incremental updates towards the solution. The complexity of a single iteration typically scales with the cost of applying $\A$ to a vector (denoted here by $\Tc_{\A}$), which is no more than $O(n^2)$ arithmetic operations for dense matrices.

However, the effectiveness of iterative methods is highly dependent on the spectral properties of $\A$. In particular, existing convergence guarantees \cite{saad2003iterative} suggest that for a matrix with condition number $\kappa(\A) = \|\A\|_2\|\A^{-1}\|_2$, as many as $O(\kappa(\A)\log(1/\epsilon))$ iterations may be needed to attain a relative error approximation using popular algorithms such as LSQR~\cite{paige1982lsqr}, or at least $O(\sqrt{\kappa(\A)}\log(1/\epsilon))$  using Conjugate Gradient (CG, \cite{hestenes1952methods}) if $\A$ is positive semidefinite (PSD). 
While these guarantees can be fine-tuned given additional information about the spectrum of $\A$ \cite{axelsson1986rate}, they suggest that the answer to our question depends on the conditioning of the input matrix. Is there no hope then of attaining \emph{universal} convergence guarantees for iterative methods, i.e., relative error bounds that converge to zero at a rate that is independent of the input?

Naturally, a proper answer to this question requires specifying what 
``relative error'' refers to. 
The majority of existing results measure the convergence of iterative methods through the notion of a relative \emph{forward error} \cite{saad2003iterative}. Given an output vector $\x\in\R^n$, its (relative) forward error with respect to a norm $\|\cdot\|$, matrix $\A$, and vector $\b$ is defined as:
\begin{align*}
    \text{(forward error)}\qquad\min_{\x_*}\quad \|\x-\x_*\|/\|\x_*\|\quad\text{s.t.}\quad \A\x_*=\b.
\end{align*}
Here, the norm $\|\cdot\|$ most commonly refers to the 2-norm $\|\cdot\|_2$, or if $\A$ is PSD, the $\A$-norm $\|\cdot\|_{\A}$. 
Equally common is what is known as the relative residual\footnote{ 
The residual is often used to measure convergence in practice, because $\|\x-\x_*\|_2$ cannot be computed without knowing $\x_*$. Backward error, by contrast, can be computed easily for linear systems (see \cite{Higham:2002:ASNA} and Lemma~\ref{l:berr-formula}) and least-squares problems~\cite{epperly2026fast}.
} $\|\A\x-\b\|_2/\|\b\|_2$; this can be viewed as the forward error
with respect to the $\A^\top\A$-norm $\|\cdot\|_{\A^\top\A}$. 
Thanks to the normalization, we can assert to have obtained a useful approximation as long as the error is sufficiently less than 1. Unfortunately, even attaining forward error $1/2$ may take an enormous number of iterations when dealing with an ill-conditioned system, and existing 
lower bounds show that 
dependence on condition number 
is in general unavoidable  among all algorithms that access $\A$ through matrix-vector products \cite{derezinski2026matrix}. 

Despite the prevalence of forward error in the convergence analysis of iterative methods, much of the numerical analysis literature has advocated for the \emph{backward error} as a more fundamental notion of approximation quality in numerical linear algebra \cite{Higham:2002:ASNA}. 
Informally speaking, 
small backward error means we have the exact solution for a slightly perturbed problem. This
 accounts for the fact that input data to any given problem has likely already undergone some distortion before reaching the algorithm, and thus it is sufficient to solve a perturbed problem instance. Given an output vector $\x\in\R^n$, its (relative) backward error with respect to a norm $\|\cdot\|$, matrix $\A$, and vector $\b$ is defined as:
\begin{align*}
    \text{(backward error)}\qquad \min_{\tilde\A}\quad \|\tilde\A-\A\|/\|\A\|\quad\text{s.t.}\quad \tilde\A\x=\b.
\end{align*}
Here, the norm $\|\cdot\|$ typically refers to the $2$-norm (which is what we will use throughout), or possibly the entrywise infinity norm, and one may also allow distortion in vector $\b$ (which we omit for simplicity).

Backward error analysis of numerical methods dates back over 60 years to classical works of Wilkinson \cite{wilkinson1960error,wilkinson1961error}, and even before that to von Neumann and Goldstine \cite{von1947numerical}. Since then, numerical linear algebra textbooks \cite{trefbau,Higham:2002:ASNA,golub2013matrix} have  argued that backward error is the most practically relevant metric of approximation quality, and popular linear system solvers such as LSQR \cite{paige1982lsqr} have been using it as the default stopping criterion 
\cite{arioli1992stopping,paige2002residual}.
Currently, with the rise of mixed-precision computations  and compressed data storage on GPU-based hardware architectures \cite{higham2022mixed}, as well as recent applications in training machine learning models \cite{barrett2020implicit,beerens2024adversarial}, backward error has an even more central role to play in the development of numerical algorithms. 

Fortunately, as an approximation metric for linear systems, backward error is provably less restrictive than forward error as it is insensitive to small perturbations of the input. This would suggest that it should enjoy better (ideally, universal) convergence guarantees for iterative linear system solvers.  Yet, perhaps surprisingly, no such guarantees are currently known in the literature.

\paragraph{Main results for PSD linear systems} In this work, we initiate the study of convergence in backward error for iterative linear system solvers, showing that when viewed from this perspective, our motivating question has a simple resolution for the  class of PSD linear systems. We start with Richardson iteration, and prove that it attains a simple universal convergence guarantee (Theorem~\ref{t:richardson}):
\smallskip\begin{center}
    \emph{After $k$ steps of Richardson iteration on a PSD linear system, \\
    the backward error is at most $1/k$.}
\end{center}\smallskip
Note that, again due to normalization, we can assert to have obtained a meaningful approximation as long as backward error is sufficiently less than 1. In particular, to attain an $\epsilon\in(0,1)$ backward error Richardson requires $\lceil1/\epsilon\rceil$ iterations and has $O((\Tc_\A+n)/\epsilon)=O(n^2/\epsilon)$ complexity.
In comparison, the standard convergence analysis yields $O(n^2\kappa(\A)\log(1/\epsilon))$ complexity which applies to both forward and backward error. Naturally, if the condition number of $\A$ is small, then the standard analysis gives an exponentially faster convergence rate, but if $\kappa(\A)$ is large, then our universal rate shows that Richardson makes substantial progress in backward error long before it makes any progress in forward error. Our proof of Theorem \ref{t:richardson} is closely connected with the common intuition that Richardson converges faster along the top eigendirections, a phenomenon which is captured much better by the backward error than it is by the forward error.

Can Richardson's $1/k$ convergence rate for PSD linear systems be improved by utilizing modern Krylov solvers? Interestingly, while methods such as MINRES and CG are not designed out-of-the-box to minimize backward error, we can impose a \emph{deliberate perturbation} in the form of regularization, replacing matrix $\A$ with $\A+\delta\I$ for an appropriately chosen $\delta$, to ensure that after $k$ steps the backward error of CG or MINRES is bounded by $O(\log^2(k)/k^2)$. However, this is still not the best universal rate attainable by a Krylov method. We show this by developing and analyzing an algorithm, MINBERR, that directly minimizes backward error over the Krylov subspace, and attains an $O(1/k^2)$ convergence rate (Theorem \ref{t:minberr}). We provide an efficient implementation of MINBERR by exploiting the structure and symmetry of the Lanczos decomposition for a PSD matrix $\A$, attaining an algorithm that takes $O((\Tc_{\A}+n)k)$ time to perform $k$ iterations. 
This allows us to further improve on the complexity guarantee attained by Richardson iteration (Corollary \ref{c:minberr}):
\smallskip\begin{center}
    \textit{We can compute an $\epsilon$ backward error solution of an $n\times n$ PSD linear system\\ in $O(\,n^2\,/\sqrt\epsilon\,)$ time.}
\end{center}\bigskip

\paragraph{Main results for general linear systems} Can universal convergence guarantees be obtained for general linear systems? One might expect that running the aforementioned PSD solvers on the normal equations of a general system should yield some (possibly weaker) universal rates in terms of backward error. Unfortunately, we show that Richardson on the normal equations can converge as slow as $\Omega(\kappa(\A)/k)$, i.e., it exhibits an initial blow-up of the backward error by up to a factor of $\kappa(\A)$, and thus cannot attain any universal convergence rate (Theorem \ref{t:failure}). We also observe a similar phenomenon empirically for Krylov solvers such as LSQR and LSMR (Figure~\ref{fig:nonsym}). 

We investigate this further by developing a variant of MINBERR which minimizes backward error over the Krylov subspace defined by the normal equations (\algne).
We show that, while this method still does not strictly speaking attain universal convergence, it gets tantalizingly close with an $O(\log(\kappa(\A))/k)$ rate (Theorem \ref{t:minberrne}). The logarithmic dependence on the condition number is indeed observed empirically on some hard instances (Figure \ref{fig:nonsym_hard}), 
but it can be entirely avoided by employing a random perturbation. We show that running \algne\ on $\A+\delta\G$ for a carefully chosen $\delta$ and a Gaussian matrix $\G$ 
nearly resolves our main motivating question in backward error for general $n\times n$ linear systems
(Corollary~\ref{c:minberrne}):
\smallskip\begin{center}
    \textit{We can compute an $\epsilon$ backward error solution of a general $n\times n$ linear system\\ in $O(\,n^2\log(n)/\epsilon\,)$ time.}    
\end{center}

\section{Background and Related Work}
\label{s:related-work}
Here, we provide additional background on backward error and convergence analysis of linear system solvers, while also highlighting an important connection to continuous optimization literature.

\paragraph{Backward error analysis} Trefethen and Bau in their Numerical Linear Algebra textbook \cite{trefbau} describe backward error analysis as ``a fundamental idea linking conditioning and stability, whose power has been proved in innumerable applications since the 1950s''.
This idea was primarily initiated by Wilkinson, whose books on ``Rounding Errors in Algebraic Processes" \cite[1963]{wilkinson1963rounding} and ``The Algebraic Eigenvalue Problem" \cite[1965]{wilkinson:1965} have achieved the status of classics, according to Higham \cite{Higham:2002:ASNA} who provides a modern treatment of the subject. 

Developing backward error analysis for iterative linear system solvers and Krylov subspace methods has been the subject of extensive work, including by Paige \cite{paige1976error} and Greenbaum \cite{greenbaum1989behavior} on the stability of CG/Lanczos, and it remains an active area of research to this day \cite{musco2018stability}. These works focus on how rounding errors in finite precision arithmetic affect the backward error relative to an exact execution of the algorithm (which is different from the true backward error, since the algorithm returns an approximate solution even in exact arithmetic). Paige and Saunders \cite{paige1982lsqr} suggested backward error tolerance as a natural stopping criterion for LSQR,\footnote{Backward error tolerance is the default stopping criterion in SciPy's LSQR implementation \cite{virtanen2020scipy}.} and later, Fong and Saunders \cite{fong2011lsmr} cited strong backward error performance as a highly desirable advantage of their LSMR algorithm. 
Kasenally and Simoncini proposed GMBACK~\cite{kasenally1995gmback,kasenally1997analysis}, which  minimizes backward error over a Krylov subspace, similarly to MINBERR and \algne, except with a much higher per-iteration cost since it does not exploit symmetry via the Lanczos decomposition. Backward error has also been recently used in machine learning to interpret model training via stochastic gradient descent as solving a nearby or modified optimization problem, providing a principled way to reason about training dynamics and implicit regularization \cite{barrett2020implicit,di2023backward, feng2019uniform}, and to formalize stability with respect to structured perturbations of data and parameters~\cite{beerens2024adversarial}.

\paragraph{Convergence analysis of iterative solvers} Even though backward error analysis has been used in the development of many iterative solvers, their convergence properties have been understood primarily in terms of forward error. This often leads to highly pessimistic guarantees where the number of iterations scales with the condition number of the problem \cite{saad2003iterative,golub2013matrix}. On the positive side, Axelsson and Lindskog \cite{axelsson1986rate} among many others have shown that by relying on polynomial approximation theory those guarantees can be improved for Krylov methods when the input matrix has clustered or isolated eigenvalues. Nevertheless, following the ideas of Nemirovsky and Yudin \cite{nemirovskij1983problem},  Chou \cite{chou1987optimality} showed that no deterministic algorithm can attain a \emph{universal} convergence rate in forward error, since the dependence on the condition number is inevitable in the worst case. Similar hardness results apply when we restrict ourselves to PSD linear systems, and they can also be extended to all randomized algorithms~\cite{derezinski2026matrix}. 

\paragraph{Connections to continuous optimization}
Our universal convergence guarantees bear some resemblance to the analysis of iterative solvers for minimizing convex objective functions, such as Gradient Descent (GD) and Nesterov's Accelerated Gradient (NAG) method \cite{nesterov1983method}. In particular, Richardson can be viewed as GD with respect to a certain convex quadratic objective (as discussed in detail in Section \ref{s:analysis}), and CG exhibits strong connections with NAG in the way it achieves acceleration \cite{d2021acceleration}.
In the continuous optimization literature, convergence of GD, NAG, and related methods is established in terms of the rate of decay of the so-called excess objective function value (relative to its minimum), 
attaining rates of the form $O(L/k)$ for GD and $O(L/k^2)$ for NAG (a.k.a.~sublinear rates, see \cite{d2021acceleration} for an overview), where $L$ is a problem-dependent parameter. Similar sublinear rates have been established for the forward error of CG by Axelsson and Kaporin \cite{axelsson2000sublinear}, where a different problem-dependent parameter $L$ appears. While none of these guarantees can be directly used to attain universal convergence of either backward or forward error in linear system solvers, the idea of using regularization to establish sublinear convergence of an algorithm (as we do in Theorem~\ref{t:perturbed-minres}) is well known in continuous optimization \cite{allen2016optimal}.

\section{Universal Convergence Analysis for Classical Solvers}
\label{s:analysis}
Here, we present universal convergence analysis of backward error for several classical iterative solvers, including Richardson iteration, Conjugate Gradient, and MINRES.
\subsection{Preliminaries}\label{sec:berr}
We start with notation and preliminaries.
\paragraph{Notation} We use $\|\cdot\|_2$ 
to denote the 2-norm (spectral norm for matrices and Euclidean norm for vectors). For a positive semidefinite (PSD) matrix $\M$, we define the $\M$-norm, $\|\x\|_{\M}:=\sqrt{\x^\top\M\x}$. We use $\kappa(\A) := \frac{\sigma_{\max}(\A)}{\sigma_{\min}(\A)}$ to denote the condition number, where $\sigma_{\max}$ and $\sigma_{\min}$ are the largest and smallest singular values, respectively.

\paragraph{Backward error} We start by formulating the standard definition of backward error for general square linear systems of the form $\A\x=\b$, where for simplicity we will assume throughout that $\A,\b\neq \zero$. Here, we consider a variant of backward error where only a perturbation of $\A$ is allowed, but not of $\b$. Naturally, if one allows perturbations of $\b$ as well, then this can only reduce the backward error, so all our convergence guarantees also implicitly apply to that variant.
\begin{definition}[Backward error]
    Given $\A\in\R^{n\times n}$ and $\b, \x\in\R^n$, define:
    \begin{align*}
        \berr_{\A,\b}(\x) \ := \
        \min_{\Delta\A}\ \frac{\|\Delta\A\|_2}{\|\A\|_2}\quad\textnormal{s.t.}\quad  (\A+\Delta\A)\x=\b.
    \end{align*}
\end{definition}
In our analysis, we rely on the classical representation of backward error for square matrices in terms of the residual error normalized by the norm of the output vector.
\begin{lemma}[Lemma 1.1, \cite{Higham:2002:ASNA}]\label{l:berr-formula}
    For any $\x\neq \zero$, its backward error with respect to $\A \in\R^{n\times n}$ and $\b \in\R^n$ satisfies:
    \begin{align*}
        \berr_{\A,\b}(\x) = \frac{\|\A\x-\b\|_2}{\|\A\|_2\|\x\|_2}.
    \end{align*}
\end{lemma}
\paragraph{Model of computation} Since our focus is on convergence rates rather than numerical stability, the theoretical analysis is performed in exact arithmetic. Nonetheless, extension of our results to finite-precision arithmetic is an important future direction. 
While this is beyond the scope of this work, we
note that
in standard floating-point arithmetic with unit roundoff $u$, it is in general not possible to obtain $O(u)$ forward error~\cite[Ch.~7]{Higham:2002:ASNA}, even in terms of the relative residual~\cite{arioli1992stopping}. By contrast, classical (and most robust) algorithms are known to be backward stable, that is, they compute a solution with $O(u)$ backward error for any given problem. Therefore, requiring a small backward error is arguably the \emph{only} sensible goal in finite precision, and hence it makes excellent sense to analyze convergence of a method in terms of the backward~error.

\subsection{Richardson Iteration}
We start by analyzing the backward error convergence of the classical Richardson iteration for positive semidefinite linear systems. The following result shows that after $k$ steps of Richardson iteration with step size $\|\A\|_2^{-1}$, the backward error is bounded by $1/k$, thus attaining universal convergence.
\begin{theorem}[Richardson iteration]\label{t:richardson}
    Given an $n\times n$ PSD $\A$ and $\b\in\R^n$, let
    \begin{align*}
         \x_0=\zero,\qquad \x_{k+1} = \x_k - \eta(\A\x_k - \b).
    \end{align*}
    Fix any $C\geq 1$ and let $\eta = \frac1{C\|\A\|_2}$. Then, for every $k\geq 1$ the iterates $\x_k$ satisfy
    \begin{align*}
        \berr_{\A,\b}(\x_k) \leq \frac Ck.
    \end{align*}
\end{theorem}
\begin{proof}
We rewrite the residual vector after $k$ iterations of Richardson as follows:
\begin{align*}
    \A\x_k- \b &= \A(\x_{k-1} - \eta(\A\x_{k-1}-\b)) - \b \\
    &= \A\x_{k-1} - \b - \eta\A(\A\x_{k-1}-\b)\\
    &= (\I - \eta\A)(\A\x_{k-1}-\b) = -(\I-\eta\A)^k\b.
\end{align*}
Next, we perform a similar calculation to derive the expression for the iterate $\x_k$:
\begin{align*}
    \x_k = (\I - \eta\A)\x_{k-1} + \eta\b = \eta\sum_{i=0}^{k-1}(\I-\eta\A)^i\b.
\end{align*}
We can now bound the backward error $\berr_{\A,\b}(\x_k)= \frac{\|\A\x_k-\b\|_2}{\|\A\|_2\|\x_k\|_2}$ as follows:
\begin{align*}
    \frac{\|\A\x_k-\b\|_2}{\|\A\|_2\|\x_k\|_2} &= \frac{C\|(\I-\eta\A)^k\b\|_2}{\|\sum_{i=0}^{k-1}(\I-\eta\A)^i\b\|_2}\leq C\bigg\|(\I-\eta\A)^k\Big(\sum_{i=0}^{k-1}(\I-\eta\A)^i\Big)^{-1}\bigg\|_2,
\end{align*}
where in the last step we use that the matrix $\sum_{i=0}^{k-1}(\I-\eta\A)^i\succeq (\I-\eta\A)^0=\I$ is invertible. 
Since all of the matrices under the spectral norm commute, letting $\lambda_i$ denote the eigenvalues of $\A$, it follows that:
\begin{align*}
    \berr_{\A,\b}(\x_k) \leq C\max_{1\leq j\leq n}\frac{(1-\eta\lambda_j)^k}{\sum_{i=0}^{k-1}(1-\eta\lambda_j)^i}.
\end{align*}
Thus, it suffices to bound the last expression by $1/k$ for each $j$. Here, we consider two cases. First, suppose that $\lambda_j=0$. Then, $(1-\eta\lambda_j)^i = 1$ for each $i$, and so the bound follows immediately. Next, suppose that $\lambda_j>0$, and let $x = \eta\lambda_j$. Using the formula for the geometric sum, we have
\begin{align*}
    \sum_{i=0}^{k-1}(1-x)^i = \frac{1-(1-x)^k}{x},
\end{align*}
so it suffices to show the following elementary inequality,
\begin{align}
\frac{x(1-x)^k}{1-(1-x)^k}\leq 1/k,\label{eq:elementary}
\end{align}
for all $0<x\leq 1$ and $k\geq 1$. This follows from Bernoulli's inequality by observing that
\begin{align*}
    (1-x)^k(1+xk)\leq (1-x)^k(1+x)^k = (1-x^2)^k\leq 1.
\end{align*}
Rearranging the terms, we obtain \eqref{eq:elementary}, thereby concluding the proof.
\end{proof}

The above result bears close resemblance to  classical convergence guarantees for gradient descent over convex quadratic objectives \cite{bertsekas2016}. In particular, when minimizing a function $\phi(\x) = \frac12\|\x-\x_*\|_{\A}^2$ for some vector $\x_*\in\R^n$ and a PSD matrix $\A$, the gradient descent update $\x_{k+1} = \x_k - \eta\nabla \phi(\x_k)$ precisely coincides with the Richardson iteration ran on $\A$ and $\b=\A\x_*$, and the textbook gradient descent analysis (e.g., \cite{lee2019first}) yields the following bound:
\begin{align}
    \frac12\|\x_k-\x_*\|_{\A}^2 = \phi(\x_k) - \phi(\x_*) \leq \frac{C}{2k}\|\A\|_2\|\x_*\|_2^2.\label{eq:gd-analysis}
\end{align}
We can convert this into a bound on the residual norm, obtaining a different type of convergence rate that slightly resembles our backward error guarantee:
\begin{align*}
    \|\A\x_k-\b\|_2 \leq \sqrt{\|\A\|_2}\|\x_k-\x_*\|_{\A}\leq \sqrt{\frac C{ k}}\|\A\|_2\|\x_*\|_2.
\end{align*}
However, this conversion is not without a cost. In fact, Theorem \ref{t:richardson} implies a strictly sharper guarantee, since due to the monotonicity of $\|\x_k\|_2$ we can upper bound the $\|\x_*\|_2$-normalized residual by the backward error:
\begin{align}
    \frac{\|\A\x_k-\b\|_2}{\|\A\|_2\|\x_*\|_2}\leq \berr_{\A,\b}(\x_k).\label{eq:connection}
\end{align}
Nevertheless, the connection to gradient descent begs the question whether  universal convergence of backward error is more generally tied to convergence of the $\|\x_*\|$-normalized residual, and with it, potentially to the broader family of tools from convex optimization theory. Unfortunately, those two notions of convergence can exhibit dramatically different behavior, as explained below.

First, for some positive semidefinite linear systems the $\|\x_*\|$-normalized residual is not even well-defined. Concretely, when $\A$ is rank-deficient and $\b$ falls outside of its range, then such a system is inconsistent and the solution $\x_*$ does not exist. Yet, even in these cases backward error is well-defined and  converges at a $1/k$ rate when running Richardson (Theorem \ref{t:richardson} does not require an $\x_*$  to exist such that $\A\x_*=\b$).

Second, we show that for non-symmetric linear systems, even when the solution does exist, the reverse direction of \eqref{eq:connection} can fail horribly, in that, the $\|\x_*\|_2$-normalized residual enjoys a conditioning-free convergence but the backward error exhibits a highly conditioning-dependent behavior. The below result shows that, if we run Richardson iteration on the normal equations, then the gap between the backward error and the $\|\x_*\|_2$-normalized residual can be as large as a factor proportional to the condition number. In particular, this suggests that the rate of increase of $\|\x_k\|_2$ plays a crucial role in the backward error convergence.

\begin{theorem}[Richardson on normal equations]\label{t:failure}
    Given an invertible $\A\in\R^{n\times n}$ and $\b\in\R^n$, let 
    \begin{align*}
        \x_0 = \zero,\qquad \x_{k+1} = \x_k - \eta\A^\top(\A\x_k-\b).
    \end{align*}
    Fix any $C\geq 1$ and let $\eta = \frac1{C\|\A\|_2^2}$. Then, denoting $\x_*=\A^{-1}\b$, for every $k\geq 1$ the iterates $\x_k$ satisfy
    \begin{align*}
        \berr_{\A,\b}(\x_k)\leq \frac{C\kappa(\A)}{k}\qquad\text{and}\qquad \frac{\|\A\x_k-\b\|_2}{\|\A\|_2\|\x_*\|_2}\le \sqrt{\frac C{2k}}.
    \end{align*}
    Furthermore, for any invertible $\A$ there exists $\b$ such that for every $1\leq k\leq \kappa^2(\A)-1$,
    \begin{align*}
        \berr_{\A,\b}(\x_k) \geq \frac{C\kappa(\A)}{ek}
        \geq \kappa(\A)\sqrt{\frac{C}{4k}}\cdot \frac{\|\A\x_k-\b\|_2}{\|\A\|_2\|\x_*\|_2}.
    \end{align*}
\end{theorem}
\begin{proof}
    Following the same argument as in the proof of Theorem \ref{t:richardson}, we can express the backward error and residual error as follows:
    \begin{align*}
        \berr_{\A,\b}(\x_k) = \frac{\sqrt{C\eta}\|\A\M_k\x_*\|_2}{\|(\I-\M_k)\x_*\|_2},\qquad \frac{\|\A\x_k-\b\|_2}{\|\A\|_2\|\x_*\|_2} = \sqrt{C\eta}\frac{\|\A\M_k\x_*\|_2}{\|\x_*\|_2},
    \end{align*}
    where  we let $\M_k = (\I - \eta\A^\top\A)^k$. Bounding the errors in terms of the singular values $\sigma_1\geq \sigma_2\geq ...\geq \sigma_n>0$ of $\A$, we obtain:
    \begin{align*}
        \berr_{\A,\b}(\x_k) \leq \sqrt{C}\,\max_i \frac{\sqrt{\eta\sigma_i^2}(1-\eta\sigma_i^2)^k}{1-(1-\eta\sigma_i^2)^k}\leq \sqrt{\frac C{\eta\sigma_n^2}}\,\max_i\frac{\eta\sigma_i^2(1-\eta\sigma_i^2)^k}{1-(1-\eta\sigma_i^2)^k}\leq C\kappa(\A)\frac1k,
    \end{align*}
    where in the last step we used \eqref{eq:elementary}. On the other hand, for the residual error, we obtain:
    \begin{align}
        \frac{\|\A\x_k-\b\|_2}{\|\A\|_2\|\x_*\|_2}\leq \sqrt C \,\max_i\sqrt{\eta\sigma_i^2(1-\eta\sigma_i^2)^{2k}}\leq \sqrt C \max_{x\geq 0} \sqrt{x\exp(-2kx)}\leq \sqrt{\frac{C}{2k}},\label{eq:residual}
    \end{align}
    where in the last step we used the elementary inequality that $x\exp(-ax)\leq 1/a$ for all $a,x>0$.

    To show that the upper bound on the backward error can be sharp up to a constant factor, let $\b$ be the left singular vector of $\A$ associated with its smallest singular value. Then, $\x^*$ is aligned with the corresponding right singular vector, so
    \begin{align*}
        \berr_{\A,\b}(\x_k) = \sqrt{C}\, \frac{\sqrt{\eta\sigma_n^2}(1-\eta\sigma_n^2)^k}{1-(1-\eta\sigma_n^2)^k}
        \geq \sqrt{C}\, \frac{\sqrt{\eta\sigma_n^2}(1-\eta\sigma_n^2)^k}{k\eta\sigma_n^2}\geq \frac{\sqrt C e^{-1}}{k\sqrt{\eta\sigma_n^2}},
    \end{align*}
    where, letting $x=\eta\sigma_n^2$, we first used Bernoulli's inequality to get $1-(1-x)^k\leq kx$, and then we used that $k\leq \kappa^2(\A)-1\leq \frac1{\eta\sigma_n^2}-1$  to obtain $(1-x)^k\geq (1-x)^{\frac1x-1} = (1+\frac x{1-x})^{-\frac{1-x}x} \geq e^{-1}$, with the last step being the standard lower bound on the constant $e$. Since $1/\sqrt{\eta\sigma_n^2}=\sqrt C\kappa(\A)$, we get:
    \begin{align*}
        \berr_{\A,\b}(\x_k) \geq \frac{C\kappa(\A)}{ek} = \kappa(\A) \sqrt{\frac{2C}{e^2k}}\sqrt{\frac{C}{2k}}.
    \end{align*}
    Applying the upper bound \eqref{eq:residual} and using that $e^2/2<4$ concludes the proof.
\end{proof}
We note that the $O(1/\sqrt k)$ upper bound on the $\|\x_*\|_2$-normalized residual can also be obtained directly from the gradient descent guarantee \eqref{eq:gd-analysis}. Also, the regime $1\leq k\leq \kappa^2(\A)-1$, where the backward error may be much larger than the $\|\x_*\|_2$-normalized residual, is natural since when $k\geq\kappa^2(\A)$, the solution norm $\|\x_k\|_2$ becomes comparable to $\|\x_*\|_2$ and both metrics start to converge at the fast exponential rate, $(1-\frac1{\kappa^2(\A)})^k$.

The failure of Richardson on normal equations, illustrated by Theorem \ref{t:failure}, perhaps further highlights the remarkable nature of the universal $1/k$ backward error rate in the positive semidefinite case (Theorem~\ref{t:richardson}). Can this rate be further improved, for instance by employing Krylov subspace methods?

\subsection{Faster Rates via Regularized Krylov Solvers}

We next show that one can obtain faster-than-Richardson universal backward error rates using existing solvers such as CG and MINRES by relying on the perturbation interpretation of backward error. The key observation is that, by the definition of backward error, we do not need to solve the original problem directly, but rather, we can solve a ``perturbed'' problem that is chosen to have better properties for our method. 

The first ingredient in our analysis is the observation that a backward error bound for one linear system can be easily converted to a bound for another system based on a nearby input matrix.
\begin{lemma}\label{l:composition}
    Given $\A,\tilde\A\in\R^{n\times n}$, $\b\in\R^n$, and $\epsilon\geq 0$ such that $\|\A-\tilde\A\|_2\leq \epsilon\|\A\|_2$, for any $\x\in\R^n$, we have
    \begin{align*}
        \berr_{\A,\b}(\x)\leq (1+\epsilon)\berr_{\tilde\A,\b}(\x) + \epsilon.
    \end{align*}
\end{lemma}
\begin{proof}
    First, observe that 
    \begin{align*}
        \|\A\x-\b\|_2 = \|\tilde\A\x-\b + (\A-\tilde\A)\x\|_2\leq 
        \|\tilde\A\x-\b\|_2 + \epsilon\|\A\|_2\|\x\|_2.
    \end{align*}
    Now, using Lemma \ref{l:berr-formula} and that $\|\tilde\A\|_2\leq (1+\epsilon)\|\A\|_2$ we obtain:
    \begin{align*}
        \berr_{\A,\b}(\x) = \frac{\|\A\x-\b\|_2}{\|\A\|_2\|\x\|_2} \leq \frac{\|\tilde\A\x-\b\|_2}{\|\A\|_2\|\x\|_2} + \epsilon
        \leq (1+\epsilon)\frac{\|\tilde\A\x-\b\|_2}{\|\tilde\A\|_2\|\x\|_2} + \epsilon,    
    \end{align*}
    which completes the proof.    
\end{proof}
Lemma \ref{l:composition} implies that if we are seeking a backward error bound for a PSD linear system $\A\x=\b$, then it suffices to analyze the error for a slightly regularized system $(\A+\epsilon\|\A\|_2\I)\x=\b$. Such regularization has the effect that it improves the condition number of the system sufficiently to ensure fast convergence of standard solvers in forward error. Then, it suffices to bound the backward error in terms of the forward error as follows.
\begin{lemma}\label{l:forward}
    Given a square consistent linear system $\A\x^*=\b$ and $\epsilon\in(0,1)$, suppose that $\x$ satisfies:
    \begin{align}
        \frac{\|\x-\x_*\|_{\M}}{\|\x_*\|_{\M}}\leq \epsilon,\quad\text{where }\M=(\A^\top\A)^\alpha\text{ for some }\alpha\in[0,1].\label{eq:forward-lemma}
    \end{align}
    Then, $\x$ satisfies the following backward error guarantee:
    \begin{align*}
    \berr_{\A,\b}(\x) \leq \frac{\epsilon}{1-\epsilon}.
    \end{align*}    
\end{lemma}
\begin{remark}
Setting $\alpha=0$ in \eqref{eq:forward-lemma} recovers a $2$-norm forward error guarantee, while $\alpha=1$ corresponds to the relative residual. If $\A$ is PSD then $\alpha=1/2$ corresponds to the $\A$-norm error. Lemma \ref{l:forward} implies that each of these forward error guarantees immediately implies a nearly matching guarantee on backward error.
\end{remark}
\begin{proof}
    Let $\A=\U\Sigmab\V^\top$ be the singular value decomposition of $\A$. Then, $\M = \V\Sigmab^{2\alpha}\V^\top$, and we can bound the residual $\|\A\x-\b\|_2$ using the $\M$-norm error as follows:
    \begin{align*}
        \|\A\x-\b\|_2 
        &= \|\U\Sigmab\V^\top(\x-\x_*)\|_2 \\
        &= \|\U\Sigmab^{1-\alpha} \V^\top\V\Sigmab^{\alpha}\V^\top(\x-\x_*)\|_2
        \\
        &\leq \|\U\Sigmab^{1-\alpha}\V^\top\|_2 \|\x-\x_*\|_{\M} = \|\A\|_2^{1-\alpha}\|\x-\x_*\|_{\M}.
    \end{align*}    
    Next, we seek to provide a lower bound on the $2$-norm of the vector $\x$. We do this by first observing that:
    \begin{align*}
        \|\x_*\|_{\M} \leq \|\x-\x_*\|_{\M} + \|\x\|_{\M} \leq \epsilon\|\x_*\|_{\M} + \|\x\|_{\M}.
    \end{align*}
    Rearranging the terms, we conclude that 
    \begin{align*}
    (1-\epsilon)\|\x_*\|_{\M}\leq \|\x\|_{\M}\leq \|\M\|^{1/2}\|\x\|_2 = \|\A\|_2^\alpha\|\x\|_2.
    \end{align*}
    Combining these with Lemma \ref{l:berr-formula}, we have:
    \begin{align*}
        \berr_{\A,\b}(\x) = \frac{\|\A\x-\b\|_2}{\|\A\|_2\|\x\|_2} \leq \frac{\|\A\|_2^{1-\alpha}\|\x-\x_*\|_{\M}}{\|\A\|_2^{1-\alpha}(1-\epsilon)\|\x_*\|_{\M}}\leq \frac{\epsilon}{1-\epsilon}.
    \end{align*}
\end{proof}

Putting these observations together, we show that by solving a slightly regularized version of the linear system, we can attain a better universal rate with both CG and MINRES than with Richardson. We note that the strategy described below can be applied to essentially any solver with a forward error convergence guarantee.

\begin{theorem}[Regularized CG/MINRES]\label{t:perturbed-minres}
    Given $n\times n$ positive semidefinite~$\A$ and $\b\in\R^n$, fix $k\geq 9$. If $\hat\x$ is produced from at least $k$ steps of CG or MINRES on input $\A+2(\frac{\log k}{k})^2\|\A\|_2\I$ and $\b$, then
    \begin{align*}
        \berr_{\A,\b}(\hat\x)\leq 
        5\Big(\frac{\log k}{k}\Big)^2.
    \end{align*}
    \end{theorem}
\begin{proof}
    Define $\epsilon = 2(\frac{\log k}{k})^2$. The condition number of the matrix $\tilde\A = \A+\epsilon\|\A\|_2\I$ can be bounded by:
    \begin{align*}
        \kappa(\tilde\A) \leq \frac{\|\A\|_2 + \epsilon\|\A\|_2}{\epsilon\|\A\|_2} = 1 + 1/\epsilon.
    \end{align*}
So, letting $\tilde\x_*=\tilde\A^{-1}\b$, CG and MINRES initialized at $\zero$ after $k$ or more iterations return $\hat\x$ such that \cite{trefbau}:
    \begin{align}
        \frac{\|\hat\x - \tilde\x_*\|_{\M}}{\|\tilde\x_*\|_{\M}}
        \leq 2\bigg(\frac{\sqrt{1+1/\epsilon} -1}{\sqrt{1+1/\epsilon}+1}\bigg)^k
        \leq 2\exp\Big(-\frac{2k}{\sqrt{1+1/\epsilon}+1}\Big),
        \label{eq:perturbed-bound}
    \end{align}
where we use $\M=\tilde\A^2$ for MINRES and $\M=\tilde\A$ for CG. 

We next show that given our choice of $\epsilon$ and $k$, the right-hand side is bounded by $\epsilon$. For this, it suffices to show that $2k\geq (\sqrt{1+1/\epsilon}+1)\log(2/\epsilon)$. It is easy to verify that $\sqrt{1+1/\epsilon}+1\leq \sqrt{2/\epsilon}$ for any $0<\epsilon\leq 1/8$, so a sufficient condition is $k\geq \sqrt{2/\epsilon}\log(\sqrt{2/\epsilon})$. By our choice of $\epsilon$ we have:
\begin{align*}
    \sqrt{2/\epsilon}\log(\sqrt{2/\epsilon}) = \frac{k}{\log k}\log\Big(\frac{k}{\log k}\Big) \leq k.
\end{align*}
In order to ensure that $\epsilon\leq 1/8$, it suffices to assume that $k\geq 9$. Thus, we have shown that the right-hand side of \eqref{eq:perturbed-bound} can be bounded by $\epsilon$. Applying Lemma \ref{l:composition}, and then Lemma \ref{l:forward}, we conclude that
\begin{align*}
    \berr_{\A,\b}(\hat\x) \leq (1+\epsilon)\berr_{\tilde\A,\b}(\hat\x) + \epsilon \leq (1+\epsilon)\frac{\epsilon}{1-\epsilon} + \epsilon = \frac{2\epsilon}{1-\epsilon}.
\end{align*}
Plugging in the expression for $\epsilon$ and using that $\epsilon\leq 1/8$, we obtain the claim.
\end{proof}

The approach given in Theorem~\ref{t:perturbed-minres} has some clear drawbacks. First, one has to commit to a target backward error at the beginning in order to choose the regularization level. To obtain a more accurate estimate, one would have to restart CG or MINRES with a different regularizer. Second, the indirect nature of this procedure inevitably incurs some overhead, which is manifested by the spurious logarithmic factors. We address all of these issues in the next section by proposing a dedicated Krylov solver for backward error.

\section{MINBERR: Optimal Solver for PSD Systems}
Here, we develop and analyze \algpsd\ (Algorithm \ref{alg:MINBERR}), a Krylov method 
for a PSD system $\A\x=\b$,
 that is explicitly designed to minimize the backward error. That is, letting $\Kc_{k}(\A,\b) = \Span\{\b,\A\b,...,\A^{k-1}\b\}$ be the Krylov subspace of rank $k$, consider
\begin{align}
  \x_{k} :=   \argmin_{\x\in \Kc_{k}(\A,\b)}\berr_{\A,\b}(\x).\label{eq:minberr}
 \end{align}       

\subsection{Convergence Analysis}
\label{sec:minberr_conv}

First, we show that the MINBERR iterate sequence \eqref{eq:minberr} attains a universal backward error convergence rate of $O(1/k^2)$.

\begin{theorem}[MINBERR]\label{t:minberr}
For any $n\times n$ PSD matrix $\A$, vector $\b\in\R^n$ and $k \ge 2$, the MINBERR iterate sequence \eqref{eq:minberr} satisfies:
\begin{align*}
    \berr_{\A,\b}(\x_k)\ \leq \frac 3{k^2-1}.
\end{align*}
\end{theorem}
\begin{remark}
    Numerical evidence suggests that the minimax optimal constant~is~$2$.
\end{remark}
\begin{proof}
    First, define $\bar\A = \A/\|\A\|_2$ and $\bar\b=\b/\|\A\|_2$. Note that $\Kc_{k}(\bar\A,\bar\b) = \Kc_{k}(\A,\b)$ and also $\berr_{\bar\A,\bar\b}(\x) = \berr_{\A,\b}(\x)$ for all $\x$, so without loss of generality we can assume that $\|\A\|_2=1$.
    
    Next, let $\Pc_{k-1}$ denote the set of all polynomials $p(x)=c_0+c_1x + \cdots +c_{k-1}x^{k-1}$ of degree at most $k-1$ with real-valued coefficients.  Then, using Lemma \ref{l:berr-formula},
    \begin{align*}
        \min_{\x\in \Kc_{k}(\A,\b)}\berr_{\A,\b}(\x) 
        &= \min_{p\in\Pc_{k-1}} \frac{\|\A p(\A)\b - \b\|_2}{\|p(\A)\b\|_2}
        \leq \min_{p\in\Pc_{k-1}}\|(\A p(\A) - \I)p(\A)^{-1}\|_2.
    \end{align*}
    Since $\A$ and $p(\A)$ commute, we can express the above in terms of $\A$'s eigenvalues $1=\lambda_1\geq ...\geq \lambda_n\ge 0$:
    \begin{align*}
        \big\|(\A p(\A) - \I)p(\A)^{-1}\big\|_2 = \max_{i} \Big|\frac{\lambda_i p(\lambda_i) - 1}{p(\lambda_i)}\Big| \leq \max_{x\in[0,1]} \Big| x - \frac1{p(x)}\Big|. 
    \end{align*}

It remains to show that the latter term is upper-bounded by $3/(k^{2}-1)$. We note that this order is tight: actually, both lower and upper bounds of $O(1/k^{2})$ can be concluded from a more general polynomial approximation result of Levin and Saff from 1988 \cite{levin1988degree}, who obtained it in a context unrelated to Krylov subspace theory:
  \begin{lemma}[Theorem 2.3, \cite{levin1988degree}]\label{l:levin1988}
        For any $\alpha>0$, there exist positive constants $B_\alpha,C_\alpha$ such that for any $k=1,2,...$, the following holds:
        \begin{align*}
            \frac{B_\alpha}{k^{2\alpha}}\leq \min_{p\in\Pc_k}\max_{x\in[0,1]}\Big|x^\alpha - \frac1{p(x)}\Big|\leq \frac{C_\alpha}{k^{2\alpha}}.
        \end{align*}
    \end{lemma}
Levin and Saff do not provide explicit values for $B_\alpha$ and $C_\alpha$. To address this, we present a direct construction of $p(x)$ which shows that $C_{\alpha}\leq 3$ when $\alpha = 1$.

\medskip

First, for $x \in [0, 1]$ and any positive integer $m$, define the shifted Chebyshev polynomial $T_m^*(x)$ as
$T_m^*(x) = T_m(2x - 1)$, where $T_m(x) = \cos(m \arccos x)$.
For any $k\geq 2$, let $\ell$ be an even integer $\ell \in \{k, k+1\}$,  and define \begin{equation}\label{eq:G} G(x) := \frac{1 - T_\ell^*(x)}{2\ell^2} \quad \text{ for } x \in [0,1].
\end{equation}
The following key facts about the function $G(x)$ follow from trigonometric identities.
\begin{lemma}[Properties of shifted Chebyshev polynomials]\label{lem:cheb}
Let $G(x)$ be defined as in \eqref{eq:G}. Then, it holds that
    $$G(x) = \frac{\sin^2(\ell \gamma)}{\ell^2} \text{ for } \gamma \in [0, \pi/2] \text{ such that } x = \sin^2(\gamma).$$
This implies that (a)
$
0\le G(x)\le x$ for every \(x\in[0,1]\),
and (b) the function \(G(x)-x\) has a root of multiplicity at least \(2\) at \(x=0\).
\end{lemma}
Lemma~\ref{lem:cheb} (proven in Appendix~\ref{app:cheb}) implies that 
\begin{equation}\label{eq:pofx}
p(x) := \frac{x - G(x)}{x^2} 
\end{equation}
is a well-defined polynomial of degree at most $\ell - 2 \le k-1$ and 
for any $x \in [0,1]$
\begin{align}\label{eq:polyinv-approx}
\left|x - \frac{1}{p(x)}\right| &= \left| x - \frac{x^2}{x - G(x)} \right| = \frac{x G(x)}{x - G(x)} = \frac{1}{\ell^2} \cdot\frac{\ell^2\sin^2(\gamma) \sin^2(\ell\gamma)}{\ell^2\sin^2(\gamma) - \sin^2(\ell\gamma)},
\end{align}
where $x = \sin^2(\gamma)$ as in Lemma~\ref{lem:cheb}. We maximize this expression using basic calculus.

\begin{lemma}[Explicit upper bound]\label{lem:sharp-trig-bound}
For every integer $\ell\ge 2$, define
\[
F_\ell(\gamma)
:=
\frac{\ell^2\sin^2(\gamma)\,\sin^2(\ell\gamma)}
{\ell^2\sin^2(\gamma)-\sin^2(\ell\gamma)},
\qquad \gamma\in[0,\pi/2].
\]
Then $F_\ell$ is well defined on $[0,\pi/2]$ after continuous extension at $\gamma=0$, and
\[
\max_{\gamma\in[0,\pi/2]} F_\ell(\gamma)
=
\frac{3\ell^2}{\ell^2-1}. 
\]
\end{lemma}
The proof is in Appendix~\ref{app:cheb}. Since $\ell\geq k$, together with \eqref{eq:polyinv-approx} this immediately implies:
\begin{align*}
\max_{x \in [0,1]}\left|x - \frac{1}{p(x)}\right| \le \frac3{\ell^2-1}\leq \frac 3{k^2-1}.
\end{align*}
\end{proof}

\subsection{Implementation and Complexity} 
Here, we describe how to compute the MINBERR solution $\x_k$ efficiently (and approximately), and show that MINBERR has complexity bounded by $O(n^2/\sqrt{\epsilon})$ to achieve $\epsilon$ backward error. 

First we show that 
finding the minimizer of the backward error in a subspace is equivalent to solving a certain generalized eigenvalue problem. Indeed, when minimizing $\berr_{\A,\b}(\x)$ over $\Kc_k(\A,\b)$, we seek to find a solution for $\A\x=\b$ within a given $k$-dimensional subspace $\mbox{span}(\Q_k)$, parameterized by $\Q_k\in\mathbb{R}^{n\times k}$ having orthonormal columns. 
 If we write $\x=\Q_k\y$, then minimizing $\berr_{\A,\b}(\x)=\frac{\|\A\x-\b\|_2}{\|\A\|_2\|\x\|_2}$ over $\x\in\Kc_k(\A,\b)$ reduces to minimizing the following generalized Rayleigh quotient:
\begin{align}\label{eq:GRQ}
\argmin_{\y \in \R^k} \frac{\|\A\Q_k\y-\b\|_2}{\|\Q_k\y\|_2}
= \argmin_{\y \in \R^k}
\frac{
\begin{bmatrix}
-1 & \y^\top     
\end{bmatrix}
\left(
\begin{bmatrix}
\b& \A\Q_k
\end{bmatrix}^\top
\begin{bmatrix}
\b& \A\Q_k 
\end{bmatrix}\right)
\begin{bmatrix}
-1 \\\y     
\end{bmatrix}
}{\begin{bmatrix}
-1 & \y^\top     
\end{bmatrix}
\left(
\begin{bmatrix}
0 & \\ & \I_k
\end{bmatrix}^\top\begin{bmatrix}
0 & \\ & \I_k
\end{bmatrix}\right)
\begin{bmatrix}
-1 \\\y     
\end{bmatrix}
}.\nonumber %
\end{align}
The value of the minimum is squared in the second expression. 
Note that if $\b$ is in the span of $\A\Q_k$ then MINBERR finds the exact solution to $\A\x=\b$, for which trivially $\berr_{\A,\b}(\x)=0$. 
Otherwise (in the usual case), 
$\begin{bmatrix}\b& \A\Q_k \end{bmatrix}$ has full column rank and so 
$\begin{bmatrix}
\b& \A\Q_k 
\end{bmatrix}^\top
\begin{bmatrix}
\b& \A\Q_k 
\end{bmatrix}$ is positive definite. In the latter case, 
minimizing such quotient is equivalent~\cite[Sec. 8.7.4]{golubbook}  to finding the smallest eigenpair of the generalized eigenvalue problem 
\begin{equation}\label{eq:bAQ}    
\begin{bmatrix}
\b& \A\Q_k 
\end{bmatrix}^\top 
\begin{bmatrix}
\b& \A\Q_k 
\end{bmatrix}
\begin{bmatrix}
-1 \\\y     
\end{bmatrix}
=\lambda\begin{bmatrix}
0 & \\ & \I_k
\end{bmatrix}^\top 
\begin{bmatrix}
0 & \\ & \I_k
\end{bmatrix}
\begin{bmatrix}
-1 \\\y     
\end{bmatrix}, 
\end{equation}
which is 
\begin{equation}\label{eq:minberrGEP}
\begin{bmatrix}
\|\b\|_2^2& \b^\top \A\Q_k \\
(\A\Q_k)^\top  \b& (\A\Q_k)^\top \A\Q_k
\end{bmatrix}
\begin{bmatrix}
-1 \\\y 
\end{bmatrix}
=\lambda\begin{bmatrix}
0 & \\ & \I_k
\end{bmatrix}
\begin{bmatrix}
-1 \\\y
\end{bmatrix}. 
\end{equation}
This is a positive (semi)definite generalized eigenvalue problem, and has real eigenvalues, with one at infinity. The smallest eigenvalue $\lambda_{\min}$, which is nonnegative, is equal to the square of the backward error of the MINBERR solution, that is, $\lambda_{\min} =  (\berr_{\A,\b}(\x_k))^2$.  

The argument so far holds for an arbitrary subspace $\Q_k$. 
We note that a similar process was described by Kasenally and Simoncini~\cite{kasenally1995gmback,kasenally1997analysis}, who take $\Q_k$ to be the Krylov subspace for a non-symmetric $\A$ found by the Arnoldi decomposition~\cite[Ch.~6]{saad2003iterative}, and solve such  generalized eigenvalue problem, which is unstructured except for the symmetry. Below we exploit the tridiagonal structure in the Lanczos decomposition 
to obtain a much more efficient algorithm when $\A$ is symmetric. Later, in Section~\ref{s:algne} we extend this approach to non-symmetric $\A$ by working in the Krylov subspace corresponding to the normal equations, obtaining \algne. 

\subsubsection{Fast implementation by exploiting symmetry}

Let us assume that $\A$ is symmetric (PSD is not strictly required). 
In this case, Krylov subspace methods start with the Lanczos decomposition~\cite[Sec.~6.6]{saad2003iterative}: 
\begin{equation}\label{eq:Lanczos}
\A\Q_k = \Q_{k+1}\T_{k}    
\end{equation}
 where 
$\Q_k=[\q_1,\ldots,\q_k]$ is $n\times k$, 
$\Q_{k+1}=[\q_1,\ldots,\q_{k+1}]$ is $n\times (k+1)$, 
both having orthonormal columns, 
and 
$\T_{k}$ is $(k+1)\times k$ symmetric tridiagonal. 
The columns of $\Q_k$ represent the orthonormal basis of the Krylov subspace at iteration $k$.
By the standard Lanczos construction
$\b = \Q_{k+1}\e_{k+1}\|\b\|_2$  where 
$\e_{k+1} = [1,0,\ldots,0]^\top \in\R^{k+1}$, 
so that $[\b\  \A\Q_k]  = \Q_{k+1}[\e_{k+1}\|\b\|_2\ \T_k]$.
Thus the equation \eqref{eq:bAQ} simplifies 
to 
\begin{equation}    \label{eq:MM}
\M^\top \M
\begin{bmatrix}
-1 \\\y
\end{bmatrix}
=\lambda 
\begin{bmatrix}
0 & \\ & \I_k
\end{bmatrix}
\begin{bmatrix}
-1 \\\y     
\end{bmatrix}, \qquad
\text{where}\quad \M = 
\begin{bmatrix}
\|\b\|_2\e_{k+1}&  \T_k
\end{bmatrix}. 
\end{equation}
$\M$ is a $(k+1)\times (k+1)$ tridiagonal and upper triangular matrix. 
It thus suffices to find the smallest eigenpair of the generalized eigenvalue problem~\eqref{eq:MM}. The next result shows how to do so efficiently and approximately. 
\begin{lemma}\label{l:quotient}
    Let $(\sigma_{\min},\v_{\min})$ be the smallest singular value and the corresponding right singular vector of the $k\times k$ matrix $\tilde\T_k$ obtained by removing the first row from $\T_k$ in~\eqref{eq:Lanczos}. Then,
    the smallest eigenpair of the generalized eigenvalue problem~\eqref{eq:MM} is given by $(\sigma_{\min}^2,\frac1\alpha\v_{\min})$, where $\alpha =-\frac{1}{\|\b\|_2}\left((\T_k)_{1,1}(\v_{\min})_1 +(\T_k)_{1,2}(\v_{\min})_2\right)$. Consequently, the MINBERR solution $\x_k$ after $k$ iterations satisfies:
    \begin{align*}
        \x_k = \frac1\alpha\Q_k\v_{\min},\qquad \berr_{\A,\b}(\x_k) = \sigma_{\min}.
    \end{align*}
    Moreover, for any $\widehat\v\in\mathbb{R}^{k}$, 
    defining 
    $\widehat\alpha=-\frac{1}{\|\b\|_2}\left((\T_k)_{1,1}(\widehat\v)_1 +(\T_k)_{1,2}(\widehat\v)_2\right)$ and 
    $\widehat\x = \frac1\alpha\Q_k\widehat\v,$ 
    we have 
    $
\berr_{\A,\b}(\widehat\x) = \|\tilde \T_k\widehat\v\|_2/\|\widehat \v\|_2$.
\end{lemma}
\begin{proof}
Since the second matrix in \eqref{eq:MM} is singular, there is an eigenvalue of \eqref{eq:MM} at infinity $\lambda_\infty=\infty$. 
As $\lambda_\infty$ is irrelevant to the MINBERR solution, 
we deflate, or remove, $\lambda_\infty$ 
by employing 
a congruence operation with respect to a nonsingular $\W$
such that $\W^\top(\M^\top \M, \begin{bmatrix}
0 & \\ & \I_k
\end{bmatrix})\W$ is simultaneously block diagonal, 
as follows: 
We choose $\W$ so that left-multiplying to $\M$ adds multiples of the first column of $\M$ to other columns so as to eliminate the other two nonzero elements in the first row of $\M$. 
This is achieved by setting 
\begin{align*}\W = 
\begin{bmatrix}
\multicolumn{2}{c}{\w_1} \\     
0 & \I_{k}
\end{bmatrix}
\in\mathbb{R}^{(k+1)\times (k+1)},\text{ where }
\w_1=\frac{1}{\|\b\|_2}\big[1, -(\T_k)_{1,1}, -(\T_k)_{1,2},0,\ldots,0\big].
\end{align*} 
We then obtain 
\begin{equation}
\label{eq:pairs}    
\W^\top \left(
(\M^\top \M), 
\begin{bmatrix}
0 & \\ & \I_k
\end{bmatrix}\right)\W = 
\left(\begin{bmatrix}
1& 0\\
 & \tilde \T_k
\end{bmatrix}
^\top 
\begin{bmatrix}
1& 0\\
 & \tilde \T_k
\end{bmatrix}, 
\begin{bmatrix}
0 & \\ & \I_k
\end{bmatrix}\right). 
\end{equation}

Since the eigenvalues of a matrix pair are invariant under congruence (or more generally equivalence) transformations~\cite[Sec.~7.7]{golub2013matrix}, 
to find the MINBERR solution it suffices to examine the eigenvalues of \eqref{eq:pairs}, where the eigenvalue $\lambda_\infty$ is decoupled from the rest. 
Further, since the second matrix in \eqref{eq:pairs} has $\I_k$ in the second block, we see that the smallest eigenvalue of the 
generalized eigenvalue problem~\eqref{eq:MM}
is equal to that of 
$\tilde \T_k^\top \tilde \T_k$
(i.e., a standard symmetric eigenvalue problem), which in turn can be reduced to the SVD of $\tilde \T_k\in\mathbb{R}^{k\times k}$; more precisely, $\sigma_{\min}(\tilde \T_k) = \berr_{\A,\b}(\x_k)$, and 
the corresponding  right singular vector $\v_{\min}$ is such that 
$\W\begin{bmatrix}
    0\\ \v_{\min}
\end{bmatrix}$
is parallel to $\begin{bmatrix}
1\\ \y    
\end{bmatrix}$; this follows from the congruence relation between~\eqref{eq:MM} and~\eqref{eq:pairs}. 
Now since $\W 
\begin{bmatrix}
0 \\ \v_{\min}    
\end{bmatrix} = \begin{bmatrix}
\alpha \\ \v_{\min}    
\end{bmatrix}$ where $\alpha = \w_1\begin{bmatrix}
    0\\ \v_{\min}
\end{bmatrix}=-\frac{1}{\|\b\|_2}\left((\T_k)_{1,1}(\v_{\min})_1 +(\T_k)_{1,2}(\v_{\min})_2\right)
$, we can compute\footnote{\label{footnotealpha0}Unless $\alpha=0$; which happens precisely when $\b$ is orthogonal to the entire subspace $\A\Q_k$, as we see by examining~\eqref{eq:GRQ}. Note that when $\A \succ 0$ this is impossible, as $\b^\top \A\b>0$ and $\frac{\b}{\|\b\|_2}$ is the first column of $\Q_k$. 
When $\A$ is PSD, the same holds unless $\A\b=0$, that is, $\b$ lies in the null space of $\A$. 
} $\y = \frac{1}{\alpha}\v_{\min}$. 
We obtain the MINBERR solution as  $\x_k=\Q_k\y$. 

For the latter statement, first 
note from~\eqref{eq:GRQ} that 
for any $\widehat\y\in\mathbb{R}^k$, letting $\widehat\x = \Q_k\widehat\y$ we have 
$\left\|\M\begin{bmatrix}
-1 \\ \widehat\y    
\end{bmatrix}\right\|_2=\berr_{\A,\b}(\widehat\x)$. 
Note also from~\eqref{eq:pairs} that $
\begin{bmatrix}
0\\ \tilde \T_k\widehat\v    
\end{bmatrix}
 =  \M\W \begin{bmatrix}
0\\ \widehat\v    
\end{bmatrix}$. Since 
$\widehat\alpha=-\frac{1}{\|\b\|_2}\left((\T_k)_{1,1}(\widehat\v)_1 +(\T_k)_{1,2}(\widehat\v)_2\right)$ is such that the first element of 
$\frac{1}{\widehat\alpha}\W\begin{bmatrix}
0\\ \widehat\v\end{bmatrix}$ is $-1$, that is, 
$\frac{1}{\widehat\alpha}\W\begin{bmatrix}
0\\ \widehat\v    
\end{bmatrix}=\begin{bmatrix}
-1 \\ \frac{1}{\widehat\alpha}\v
\end{bmatrix}$,  we obtain the desired result by taking $\widehat\y=\frac{1}{\widehat\alpha}\v$. 
\end{proof}

This shows that the MINBERR solution can be found via the SVD of $\tilde\T_k$, and that
one can find an approximate minimizer of $\berr_{\A,\b}(\x)$ over $\Q_k$ by finding a unit-norm vector $\v$ that approximately minimizes $\|\tilde\T_k\v\|_2/\|\v\|_2$. 
Let us consider a common situation where a prescribed tolerance $\epsilon$ is required. 
The MINBERR computation therefore proceeds as follows: First, run the Krylov iterations until $\sigma_{\min}(\tilde \T_k)<\epsilon$. Then 
find the corresponding singular vector $\v_{\min}$. 

In the first phase, given a required tolerance $\epsilon$ for the backward error, our algorithm first tests if 
$\sigma_{\min}(\tilde \T_k)$ is smaller than $\epsilon$ by attempting a Cholesky factorization of the symmetric pentadiagonal matrix
$\tilde\T_k^\top \tilde \T_k-\epsilon^2\I_k=\bf{R}^\top\bf{R}$. Cholesky succeeds if and only if $\sigma_{\min}(\tilde \T_k)>\epsilon$ holds~\cite[Ch.~10]{Higham:2002:ASNA}\footnote{The process is backward stable~\cite{Higham:2002:ASNA}; 
nonetheless, breakdown of Cholesky can occur if $\tilde G-\epsilon^2 \I_k$ is numerically rank deficient; which can happen even when $\sigma_{\min}(\tilde\T)>\epsilon$ if $\epsilon^2=O(u)$, where $u$ is the unit roundoff. This can only happen when we require backward error smaller than $\sqrt{u}$. In this case, one is advised to go to step 7 and compute the backward error by inverse iteration.}. 
Note that, at each iteration $k$, $\tilde \T_k$ grows in size by $1$, and only its final column needs computing (more specifically the bottom two elements). Similarly, the Gram matrix $\tilde \G=\tilde\T_k^\top \tilde \T_k$ and the Cholesky factor $\bf{R}$ both need to be computed only in the final columns, resulting in $O(1)$ cost overall for the convergence check. 

If Cholesky breaks down, implying $\sigma_{\min}(\tilde \T_k)<\epsilon$, 
 we terminate the Krylov iterations and proceed to compute $\v_{\min}$.
 This can be done by running the inverse iteration, 
i.e., 
 set $\v$ to be a random initial guess, and repeat: 
$\v:=(\tilde \T_k^\top \tilde \T_k)^{-1}\v$, and $\v:=\v/\|\v\|_2$.  
We then have $\v\rightarrow \v_{\min}$ as we run more iterations. 
Each application of $(\tilde \T_k^\top \tilde \T_k)^{-1}$ can be done by solving two linear systems with respect to $\tilde \T_k^\top$ and $\tilde \T_k$. These are tridiagonal and triangular, so each linear solve is $O(k)$ operations.  In practice, we observe that the inverse iteration typically converges to high accuracy in one or two steps for large $k$, which is unsurprising, since it is applied to a nearly-singular matrix with smallest singular value $<\epsilon$. 

We summarize the implementation in Algorithm~\ref{alg:MINBERR}.

\begin{algorithm}[hbtp]
  \caption{MINBERR: Solver for symmetric positive semidefinite linear systems}
  \label{alg:MINBERR}
  \begin{algorithmic}[1]
  \Require{$\A$, $\b$, iterations $k_{\max}$, backward error tolerance $\epsilon$, 
  failure probability $\delta$ 
  }
  \Ensure{
$\widehat\x_k$ such that either $k=k_{\max}$ or 
 $\berr_{\A,\b}(\x_k)< \epsilon$ }
 \State Set $\q_1 = \b/\|\b\|_2, \q_0 = 0, \beta_1 = 1, \ell = 0$
 \For{$k = 1,2,...,k_{\max}$}
    \State 
    (Lanczos:) $\hat\w_k = \A\q_k - \beta_k\q_{k-1}$, 
    $\alpha_k=\hat\w_k^\top \q_k$,
    $\w_k = \hat\w_k-\alpha_k\q_k$, 
    $\beta_{k+1}=\|\w_k\|_2,$ $\q_{k+1}=\w_k/\beta_{k+1}$.
    This executes the Lanczos process~\cite[Sec.~6.6]{saad2003iterative}, giving 
    $\A\Q_k=\Q_{k+1}\T$, where 
    $\T\in\mathbb{R}^{(k+1)\times k}$ is symmetric tridiagonal with $\T_{i,i}=\alpha_i, \T_{i,i+1}=\beta_{i+1}$. 
    \State Update 
    $\tilde\T := \footnotesize\begin{bmatrix} 
 \beta_2& \alpha_2 & \beta_3       \\
 & \beta_3 &\alpha_3 & \beta_4       \\
 && \ddots &\ddots & \ddots       \\
 &&& \ddots &\ddots & \beta_k     \\
 &&& & \beta_{k} & \alpha_{k}       \\
 &&& & & \beta_{k+1}       \\
    \end{bmatrix}$\normalsize by appending the last column.
Also update $\tilde \G=\tilde\T^\top \tilde\T$ by appending the last row and column.
\State (Check for $\epsilon$ convergence:) Attempt Cholesky factorization of $\tilde \G-\epsilon^2 \I_k = \bf{R}^\top \bf{R}$ (by computing the three nonzero entries in $\bf{R}$'s last column); if this breaks down (i.e., $\sigma_{\min}(\tilde\T)=\berr_{\A,\b}(\x_k)<\epsilon$), proceed to step 7. 
Otherwise return to step 2.
\EndFor         
\State  Find smallest singular vector of $\tilde\T$ by inverse iteration: Let $\v$ be Gaussian, and
\While{
$\ell<2.23 \ln(k / \delta^2)$}
 \State Solve $\tilde\T^\top\z = \v$ for $\z $, and $\tilde\T \v = \z$ for $\v$.
 \State Set $\v:=\v/\|\v\|_2$, \quad $\ell: = \ell+1$. 
 \EndWhile
\State $\widehat\x_k = \frac{1}{\alpha}\Q \v$ where 
$\alpha = -\frac{1}{\|\b\|_2}\left((\T_k)_{1,1}(\v)_1 +(\T_k)_{1,2}(\v)_2\right)$.
  \end{algorithmic}
\end{algorithm}
 
\vspace{-3mm}
\subsubsection{Time Complexity Analysis}

Next, we put together the above analysis to obtain the overall computational cost of Algorithm \ref{alg:MINBERR}, and show that it does in fact approximately minimize backward error in the Krylov subspace.

\begin{theorem}\label{thm:minberrcost}
Let $\A\succ 0$ be $n\times n$ and $\b\in\mathbb{R}^n$. 
Running $k$ steps of MINBERR (Algorithm~\ref{alg:MINBERR}) 
yields $\widehat\x_k$ such that $\berr_{\A,\b}(\widehat\x_k)\leq 
(1+\frac{1}{2})\min_{\x\in \Kc_{k}(\A,\b)}\berr_{\A,\b}(\x)$
with probability at least $1-\delta$, and 
takes $O\left(k(\mathcal{T}_\A+n+\log(\frac{1}{\delta})) \right)$ operations.
\end{theorem}
\begin{proof}
We look at the costs for each step of Algorithm~\ref{alg:MINBERR}. 
Step 3: one clearly needs $k\Tc_{\A}$ time for the $k$ matrix-vector multiplications with $\A$. The $O(k)$ vector inner-products required to form the Lanczos vectors amount to $O(nk)$ operations. 
Step 4 is $O(1)$ cost, as there are only three entries to be added to $\tilde\T$ and $\tilde \G$. 
The same holds for step 5, as the only quantity that needs computing is the final column of $\bf{R}$, which has three nonzeros. 

We now examine the cost of approximately computing  $\v_{\min}$ using inverse iteration, steps 7--11. 
To do so, we need to combine Lemma \ref{l:quotient} with a result on the number of inverse iterations required to obtain a 
\emph{gap-independent} constant-factor relative bound on the Rayleigh quotient for the smallest eigenvalue of a positive definite matrix  (here, $\tilde \T^\top\tilde \T$). This follows analogously to the classical analysis of Kuczyński and Woźniakowski~\cite{kuczynski1992estimating} for standard power iteration (see Appendix~\ref{a:interse-iteration}).
\begin{lemma}\label{lemma:invit}
Let $\M$ be $n\times n$ positive definite with smallest eigenvalue~$\lambda_{\min}$,
and let ${\bf v}_0$ consist of $n$ Gaussian entries. Define
${\bf v}_{\ell} = (\M^{-1})^{\ell} {\bf v}_0$. If 
$\ell\geq \frac{1}{2} \Big[ 1 + \frac{\ln\left( \frac{32n \ln(2/\delta)}{\pi \delta^2 \Delta} \right)}{\ln(1 + \Delta/2)} \Big]$ for $\Delta\in(0,1)$,
then with probability $1-\delta$ 
we have  
${\bf v}_{\ell}^\top \M {\bf v}_{\ell}/\|{\bf v}_{\ell}\|_2^2\leq (1+\Delta)\lambda_{\min}$.
\end{lemma}
Applying Lemma \ref{lemma:invit} to $\tilde \T^\top\tilde \T$ shows that $\ell \geq 2.23 \ln(k / \delta^2)$
 steps of inverse iteration suffice to get $\Delta=1/2$ relative accuracy with probability at least $1-\delta$. 

Finally, computing $\x_k=\frac{1}{\alpha} \Q_k\v$ costs $O(nk)$ operations. 
Thus the total operation count is $O\left(k(\mathcal{T}_\A+n+\log(\frac{1}{\delta})+\log k )\right) = 
O\left(k(\mathcal{T}_\A+n+\log(\frac{1}{\delta}) )\right)$, as  $k\leq n$. 
\end{proof}

Let us discuss the case when $\A$ is PSD but singular, as the convergence theory in Theorem~\ref{t:minberr} allows it. 
In this case $\tilde \T_k$ is full rank until $k$ is equal to the number of nonzero distinct components $\tilde n$ in the eigenvector expansion $\b=\sum_{i=0}^n c_i\v_i$, and the proof remains unaffected. Once $k=\tilde n$, $\Q_k$ contains a null vector of $\A$, 
which can be scaled arbitrarily large to yield an arbitrarily small backward error. 
We note that inverse iteration would then break down in exact arithmetic, but in finite precision converges with excellent accuracy~\cite{peters1979inverse}, often in one step. Alternatively, one can run MINBERR on $\A+\epsilon\|\A\|_2\I\succ \zero$ and rely on Lemma~\ref{l:composition} to bound the backward error.

By combining the $O(1/k^2)$ convergence guarantee from Theorem \ref{t:minberr} with the complexity analysis in Theorem \ref{thm:minberrcost}, we immediately obtain the following.

\begin{corollary}\label{c:minberr}
Let $\A\succeq \zero$ be $n\times n$ and $\b\in\mathbb{R}^n$. 
MINBERR finds a solution with $\epsilon$ backward error in 
$O(\frac{1}{\sqrt{\epsilon}}
\left(\mathcal{T}_\A+n+\log(\frac{1}{\delta})\right)$ time with probability at least $1-\delta$. 
\end{corollary}
\begin{remark}
    Under the very mild assumption that $\mathcal{T}_\A\geq n$, this implies that we can find an $\epsilon$ backward error solution with probability $1- 2^{-Cn}$   in $O(\mathcal{T}_{\A}/\sqrt\epsilon)$ time.
\end{remark}
We note that other natural variants of MINBERR may be considered:
\begin{itemize}
 \item If we fix the number of iterations upfront (e.g., to be $\lceil\sqrt{2/\epsilon}\rceil$, which guarantees that backward error is bounded by $\epsilon$), then we can bypass steps 4--5. 
\item In our experiments, to track the convergence we ran steps 7--12 at every iteration, which incurs an additional cost of $O(k^2\log(1/\delta))$.
\end{itemize}

\section{\algne: Extension to Non-Symmetric Systems}
\label{s:algne}
There are two approaches one may consider for extending MINBERR to non-symmetric systems: GMRES-style (standard Krylov), and LSQR-style (Krylov for the normal equations). 

With the standard Krylov subspace
$\mbox{span}(\b,\A \b,\A^2 \b,...,\A^{k-1}\b)$, a well-known example, also discussed by Kasenally~\cite{kasenally1995gmback},
\begin{align*}\A=\begin{bmatrix}
& \I_{n-1}    \\
1& 
\end{bmatrix}
\quad\text{ and }\quad\b=[0,0,\ldots,0,1]^\top
\end{align*}
gives 
a subspace 
orthogonal to the solution $\x=[1,0,\ldots,0,0]^\top$ until $k=n$, and hence backward error $\geq 1$.
This example shows that a universal convergence rate is not possible with the standard Krylov subspace. 
More generally, the convergence of GMRES is still not fully understood~\cite{greenbaumbook,greenbaum1996any}, and not fully characterized by the eigenvalues or singular values.
For these reasons here we opt for the Krylov subspace of the normal equations, 
$\Kc_k(\A^\top\A,\A^\top\b)=\mbox{span}(\A^\top \b,(\A^\top \A)\A^\top \b,\ldots, (\A^\top \A)^{k-1}\A^\top \b)$, leading to the following variant of MINBERR over normal equations (\algne):
\begin{align}
  \x_{k} \ :=   \!\!\argmin_{\x\in \Kc_{k}(\A^\top\A,\A^\top\b)}\!\berr_{\A,\b}(\x).\label{eq:minberrne}
 \end{align}       

\subsection{Convergence Analysis} 
Here, we show that the sequence \eqref{eq:minberrne} nearly achieves an $O(1/k)$ universal backward error rate, up to a mild logarithmic factor.
\begin{theorem}[\algne]\label{t:minberrne} For invertible $\A\in\R^{n\times n}$, $\b\in\R^n$, and $k\geq 2$, the sequence \eqref{eq:minberrne} satisfies:
\begin{align*}
    \berr_{\A,\b}(\x_k) \leq \frac{3 \log \kappa(\A)}{k}.
\end{align*}
\end{theorem}
\begin{remark}
     Our experiments in Section~\ref{s:experiments} suggest that the logarithmic condition number dependence is likely unavoidable but the constant 3 can likely be improved. 
\end{remark}
\begin{proof}
First, note that for $k$ larger than $\frac{1}{2}\kappa \log{\kappa}$, where we let $\kappa := \kappa(\A)$, the required bound follows from Lemma~\ref{l:forward} and using forward error of the standard Krylov methods such as CG on normal equations. Indeed, letting $\x_* = \A^{-1}\b$ be the optimum and $\x_k^{\rm cg}\in \Kc_{k}(\A^\top\A,\A^\top\b)$ be the $k$-the iterate of CG on $\A^\top\A\x=\A^\top\b$,
\[
\frac{\|\x_k^{\rm cg}-\x_*\|_{\A^\top \A}}{\|\x_*\|_{\A^\top \A}}
\le
2\left(\frac{\kappa-1}{\kappa+1}\right)^k,
\]
so by Lemma~\ref{l:forward} we have 
$$\berr_{\A,\b}(\x_k) \le \berr_{\A,\b}(\x_k^{\rm cg}) \le \frac{2(\frac{\kappa-1}{\kappa+1})^k}{1 - 2(\frac{\kappa-1}{\kappa+1})^k} \le \frac{3\log\kappa}{k},$$
where the last step follows via basic calculus for all $\kappa \ge 1$ and $k \ge \max\{2,\frac{1}{2}\kappa \log{\kappa}\}$. 
Now, let us proceed as in the psd case above by assuming without loss of generality  that $\|\A\|_2=1$ so that $\sigma_{\min}(\A) = 1/\kappa$. 
With $\sigma_i$ denoting $i$-th singular value of~$\A$,
    \begin{align*}
        \min_{\x\in \Kc_{k}(\A^\top\A,\A^\top\b)}\berr_{\A,\b}(\x) 
        &= \min_{p\in\Pc_{k-1}} \frac{\|\A p(\A^\top\A)\A^\top\b - \b\|_2}{\|p(\A^\top\A)\A^\top\b\|_2} \\
        &= \min_{p\in\Pc_{k-1}} \max_{i} \Big|\frac{\sigma_i^2 p(\sigma_i^2)-1}{\sigma_i p(\sigma_i^2)} \Big| \\
        &\leq \min_{p\in\Pc_{k-1}}  \max_{x\in[1/\kappa^2,1]} \Big| \frac{x\, p(x) - 1}{\sqrt x\, p(x)}\Big|. 
    \end{align*}
    For convenience, consider a polynomial $q(x) := 1 - x p(x)$ and note that it is enough to find a degree $k$ polynomial $q(x)$ such that $q(0) = 1$ and 
\begin{equation}\label{eq:minberne-target}
        x\left(\frac{q(x)}{1 - q(x)}\right)^2 \le t_{k,\kappa} := \left(\frac{3 \log \kappa}{k}\right)^2 \qquad \text {for all } \; x \in [1/\kappa^2, 1].
    \end{equation}
    First, if $t_{\kappa,k} \ge 1$ then we can take $q(x) := 1 - \kappa^2 x$ which is $\leq 0$ for $x \ge 1/\kappa^2$, so
$$
  x\left(\frac{q(x)}{1 - q(x)}\right)^2 \le  x\left(\frac{|q(x)|}{1 + |q(x)|}\right)^2 \le x \le 1 \le t_{\kappa,k}, $$
  to satisfy \eqref{eq:minberne-target}. So, without loss of generality, let us further assume that $t := t_{\kappa, k} < 1$, 
  i.e., $k> 3\log \kappa$. For the rest of the proof, we verify \eqref{eq:minberne-target} for all $3\log\kappa < k \le \frac{1}{2} \kappa \log(\kappa)$.

Let $T_m(x) = \cos(m \arccos x)$ be the Chebyshev polynomial of degree $m$ and define
    $$
    q(x) := (1 - \kappa^2 x) \frac{T_{k-1}(\frac{1 + t - 2x}{1 - t})}{T_{k-1}(\frac{1 + t}{1 - t})}.$$
Note that \(q(0)=1\), and \(\deg(q)\le k\). For $x\le t$, we have that $x_1 := (1 + t - 2x) /(1-t)$ and $x_2 := (1 + t) /(1-t)$ satisfy $1\le x_1\le x_2$, so $T_{k-1}$ is non-negative increasing from $x_1$ to $x_2$. Thus, for all $x \in [1/\kappa^2, t]$, we have $q(x) \le 0$ and
$$
  x\left(\frac{q(x)}{1 - q(x)}\right)^2 \le  t\left(\frac{|q(x)|}{1 + |q(x)|}\right)^2 \le t,$$
  thereby verifying \eqref{eq:minberne-target} in this sub-range.

  Further, if $x \in (t, 1],$ then $x_1 \in [-1,1] $ and so $|T_{k-1}(x_1)| \le 1$. Then, using the standard Chebyshev identity
\[
T_m\left(\frac{1+r^2}{1-r^2}\right)
=
\frac12\left[
\left(\frac{1+r}{1-r}\right)^m
+
\left(\frac{1-r}{1+r}\right)^m
\right],
\]
with \(r=\sqrt t\), we obtain
\[
\begin{aligned}
T_{k-1}\left(\frac{1+t}{1-t}\right)
\ge
\frac12
\left(\frac{1+\sqrt t}{1-\sqrt t}\right)^{k-1} &\ge
\frac12 \exp\!\left(2(k-1)\sqrt t\right) 
=
\frac12 \kappa^{6(1-1/k)}
\ge
\frac{\kappa^6}{2e^2}.
\end{aligned}
\]
Then, for $\kappa \ge 6$,
$$|q(x)| \le 2\kappa^2 x\frac{2e^2}{\kappa^6} \le \frac{2}{\kappa},$$
and so,
$$
x\left(\frac{q(x)}{1-q(x)}\right)^2
\le \left(\frac{3}{\kappa}\right)^2
\le \left(\frac{3 \log \kappa}{k}\right)^2,
\]
where we used $x \le 1$ and $k\le \frac12\kappa\log\kappa$. Observe that  $3\log\kappa < k \le \frac{1}{2} \kappa \log(\kappa)$ implies \(\kappa>6\), so there are no additional restrictions on the condition number.
\end{proof}
For highly ill-conditioned problems, the logarithmic dependence of the convergence rate on the condition number of $\A$ can be replaced by a logarithmic dependence on its dimension $n$.  This is achieved by relying on classical results from smoothed analysis \cite{sankar2006smoothed}, which show that perturbing the input with small Gaussian noise reduces the condition number to polynomial in $n$. 
\begin{theorem}[Perturbed \algne]\label{t:perturbed-minberrne}
Given $\A\in\R^{n\times n}$, $\b\in\R^n$, and $\epsilon,\delta\in(0,1/3)$, if $\G\in\R^{n\times n}$ has independent standard Gaussian entries, then 
\begin{align}
  \tilde\x_{k} \ :=   \!\!\argmin_{\x\in \Kc_{k}(\tilde\A^\top\tilde\A,\tilde\A^\top\b)}\!\berr_{\tilde\A,\b}(\x),\quad\text{where}\quad\tilde\A = \A+ \tfrac{\epsilon\|\A\|_2}{3\sqrt n}\G,\label{eq:perturbed-minberrne}
\end{align}       
with probability $1-\delta-e^{-n/2}$ for all $k\ge 2$  satisfies:
\begin{align*}
    \berr_{\A,\b}(\tilde\x_k) \leq \frac{4\log(10n/\epsilon\delta)}{k} + \epsilon.
\end{align*}
\end{theorem}
\begin{remark}
    With $\epsilon=1/n$, \eqref{eq:perturbed-minberrne} attains an $O(\log(n)/k)$ rate for all $1\le k\le n$.
\end{remark}
\begin{proof}
    Using standard norm bounds for Gaussian matrices  \cite{davidson2001local}, $\|\G\|_2\leq 3\sqrt n$ with probability $1-e^{-n/2}$, so $\|\tilde\A-\A\|_2\leq \epsilon\|\A\|_2$.
    Using Theorem~3.3 of \cite{sankar2006smoothed}, with probability $1-\delta$ we can bound the spectral norm of $\tilde\A^{-1}$ (letting $\alpha=\frac{\epsilon\|\A\|_2}{3\sqrt n}$) as:
    \begin{align*}
        \|\tilde\A^{-1}\|_2 = \alpha^{-1}\|(\alpha^{-1}\A+\G)^{-1}\|_2\leq \alpha^{-1}2.35\sqrt n/\delta = \frac{7.05\, n}{\epsilon\delta\|\A\|_2}.
    \end{align*}
    Thus, with probability $1-\delta-e^{-n/2}$ we have $\kappa(\tilde\A)\leq(1+\epsilon)7.05n/\epsilon\delta\leq 10n/\epsilon\delta$. Combining Theorem \ref{t:minberrne} with Lemma \ref{l:composition} concludes the proof.
\end{proof}

\subsection{Implementation and Complexity} 
Our starting point for the implementation of \algne\  \eqref{eq:minberrne} is the standard bidiagonalization-based decomposition as implemented for example in LSQR \cite{paige1982lsqr}:
\begin{equation}\label{eq:bidiag}
\A\Q_k=\U_k\B_k,     
\end{equation}
 where $\Q_k$ is the $n\times k$ orthonormal basis for $\Kc_k(\A^\top\A,\A^\top\b)$, $\U_k$ is $n\times (k+1)$ orthonormal with first column $\b/\|\b\|_2$, and $\B_k$ is $(k+1)\times k$ lower bidiagonal.
This requires $k$ matrix-vector multiplications with $\A$ and $\A^\top$, plus the orthogonalization cost of $O(nk)$, since it needs to be done only against the current and last vectors. 
This also yields the Lanczos decomposition for $\A^\top \A$ via
$\A^\top \A\Q_k=\Q_k\T_k + \q_{k+1}[0,\ldots,\gamma_{k+1}] $ where $\T_k=\B_k^\top \B_k$ is tridiagonal, and $\gamma_{k+1}=\alpha_k\beta_{k+1}$. 

One can plug these identities into \eqref{eq:minberrGEP} to derive an algorithm analogous to MINBERR (see Algorithm \ref{alg:MINBERRnonsym} below). Remarkably, the bidiagonal structure allows us to obtain an even more stable and efficient implementation. 
\begin{algorithm}[H]
  \caption{\algne: Solver for non-symmetric linear systems}
  \label{alg:MINBERRnonsym}
  \begin{algorithmic}[1]
  \Require{$\A$, $\b$, iterations $k_{\max}$, backward error tolerance $\epsilon$, failure probability $\delta$}
  \Ensure{
$\x_k$ such that either $k=k_{\max}$ or 
 $\berr_{\A,\b}(\x_k)< \epsilon$ }
 \State $\beta_1 = \|\b\|_2$, $\u_1=\b/\beta_1$, 
 $\alpha_1 \q_1= \A^\top \u_1$ s.t. $\|\u_1\|_2=\|\q_1\|_2=1$. 
 \For{$k = 1,2,...,k_{\max}$}
    \State 
    (Bidiagonalize:) 
    Compute $\beta_{k+1}\u_{k+1}=\A\q_k-\alpha_k\u_k$, and 
    $\alpha_{k+1}\q_{k+1} = \A^\top\u_{k+1}-\beta_{k+1}\q_k$, s.t. 
    $\|\u_{k+1}\|_2=\|\q_{k+1}\|_2=1$.
    This yields 
    $\A\Q_k=\U_k\B_k$.
\State (One dqds step to test $\epsilon$-convergence:) 
Set $d:=p_1-\epsilon^2$ (only when $k=1$).
Let $p_k=(\B_k)^2_{k+1,k}, e_{k-1}=(\B_k)^2_{k,k}$. 
Set 
$\hat p_{k-1}=d+e_{k-1}$, 
and $d:=d\cdot p_{k}/\hat p_{k-1}-\epsilon^2$. 
\State If $d>0$, return to step 2. Else, $\sigma_{\min}(\tilde \B_k)=\berr_{\A,\b}(\x_k)<\epsilon,$ so go to step 7. 
 \EndFor    
 
 \State Set $\tilde\B_k$ as the bottom $k\times k$ submatrix of $\B_k$. Let $\v$ be Gaussian.
 \While{ 
$\ell<2.23 \ln(k / \delta^2)$}
 \State Solve $\tilde\B_k^\top\z = \v$ for $\z $, and $\tilde\B_k \v = \z$ for $\v$.
 \State Set $\v:=\v/\|\v\|_2$, \quad 
 $\ell:=\ell+1$.
 \EndWhile
\State $\x_k = \frac{1}{\alpha}\Q \v$ where 
$\alpha = -\frac{1}{\|\b\|_2}(\B_k)_{1,1}(\v)_1.$
  \end{algorithmic}
\end{algorithm}

\begin{proposition}\label{prop:minberrNE}
        Let $(\sigma_{\min},\v_{\min})$ be the smallest singular value and the corresponding right singular vector of the $k\times k$ bidiagonal matrix $\tilde\B_k$ obtained by removing the first row from $\B_k$ in~\eqref{eq:bidiag}. 
        Then,
    defining $\alpha = -\frac{1}{\|\b\|_2}(\B_k)_{1,1}(\v_{\min})_1 $, 
 the MINBERR-NE solution after $k$ iterations satisfies
    \begin{align*}
        \x_k :=\!\argmin_{\x\in \Span(\Q_k)}\!\!\berr_{\A,\b}(\x) \
         = \ \frac1\alpha\Q_k\v_{\min},\qquad \berr_{\A,\b}(\x_k) = \sigma_{\min}.
    \end{align*}
    Moreover, for any $\widehat\v\in\mathbb{R}^{n\times k}$, 
    defining\footnote{Unlike the PSD case, $\alpha=0$ is possible here; in this case, the \algne\ solution does not exist, with the infimum of $\berr_{\A,\b}$ being $1$, where $\x$ approaches 
an 'infinite' multiple of a vector parallel to the smallest singular vector of $\A\Q_k$ with a residual norm $\gg \|\b\|_2$, which is hardly useful. } 
    $\widehat\alpha = -\frac{1}{\|\b\|_2}(\B_k)_{1,1}(\v)_1 $ 
     and 
    $\widehat\x = \frac{1}{\widehat\alpha}\Q_k\widehat\v,$
    we have 
    $
\berr_{\A,\b}(\widehat\x) = \|\tilde \B_k\widehat\v\|_2/\|\widehat \v\|_2$.
\end{proposition}
The proof of Proposition \ref{prop:minberrNE} is similar to Lemma~\ref{l:quotient} and 
deferred to Appendix \ref{a:minberr-ne}. 

A very efficient way to check if $\sigma_{\min}(\tilde\B_k)<\epsilon$ for a prescribed tolerance $\epsilon$ is to run a single step of the dqds algorithm~\cite{dqdsori} with shift $\epsilon$, which is equivalent to performing the Cholesky factorization but can preserve relative high accuracy. 

Once convergence is detected so that $\sigma_{\min}(\tilde\B_k)<\epsilon$, we terminate the \algne\ iterations and compute the singular vector corresponding to the smallest singular value of $\tilde\B_k$. As for $\A\succ 0$, we can do this by inverse iteration with respect to $\tilde\B_k^\top\tilde\B_k$. To compute $(\tilde\B_k^\top \tilde\B_k)^{-1}\v$ we solve $\tilde\B_k^\top\z = \v$ for $\z $, then $\tilde\B_k \v = \z$; each is a bidiagonal linear system that can be solved in $O(k)$ operations. 
Finally, as before, once $\v_{\min}$ is (approximately) obtained, one can find the \algne\ solution  $\x_k = \frac{1}{\alpha}\Q_k\v_{\min}$, where 
$\alpha = -\frac{1}{\|\b\|_2}(\B_k)_{1,1}(\v_{\min})_1 $. Altogether, we obtain the following  guarantee, where $\mathcal{T}_{\A^\top}$ denotes the cost for multiplying a vector by $\A^\top$.

\begin{theorem}\label{thm:minberrnecost}
Let $\A\in\R^{n\times n}$ be invertible, $\b\in\mathbb{R}^n$, and $\x_k$ be as in \eqref{eq:minberrne}. 
Running $k$ steps of \algne\ (Algorithm~\ref{alg:MINBERRnonsym}) 
yields $\widehat\x_k$ so that with probability $1-\delta$, $\berr_{\A,\b}(\widehat\x_k)\leq 
(1+\frac{1}{2})\berr_{\A,\b}(\x_k)$ using  
$O\left(k(\mathcal{T}_\A+\mathcal{T}_{\A^\top}+n+\log(\frac{1}{\delta})) \right)$ operations.
\end{theorem}
\begin{proof}
We examine the costs for each step of Algorithm~\ref{alg:MINBERRnonsym}. Step 1 needs one $\A^\top$-multiplication, and $O(n)$ operations. Step 3 requires $k$ multiplications each with $\A$ and $\A^\top$, along with inner-product calculations requiring $O(kn)$ operations. Step 4 is a single dqds iteration, so $O(k)$ overall. As for Steps 9--10, the same convergence guarantee as in Lemma~\ref{lemma:invit} applies, so after $\ell =O(\ln(k/\delta))$ iterations of inverse iterations we obtain a computed $\widehat\x_k$ with 
$1/2$-relative accuracy approximation of $\berr_{\A,\b}(\x_k)$. The operation count here is $\ell(\mathcal{T}_{\A}+\mathcal{T}_{\A^\top}+k)$. 
\end{proof}

Combining this with the $O(\log(n)/k)$ convergence guarantee (Theorem~\ref{t:perturbed-minberrne}), which ensures $\epsilon$ backward error in $k=O(\log(n)/\epsilon)$ iterations, gives the following; here we replace $\mathcal{T}_{\A}$ and $\mathcal{T}_{\A^\top}$ with $n^2$ as the Gaussian perturbation forces $\A$ to be dense.
\begin{corollary}\label{c:minberrne}
    Given $\A\in\R^{n\times n}$ and $\b\in\R^n$, perturbed \algne\ finds a solution with $\epsilon$ backward error in 
    $O(n^2\log(n/\delta)/\epsilon)$ 
    time with probability at least $1-\delta$.
\end{corollary}

\begin{figure}[t]
\begin{center} \begin{subfigure}{0.32\textwidth}
     \includegraphics[width=\textwidth]{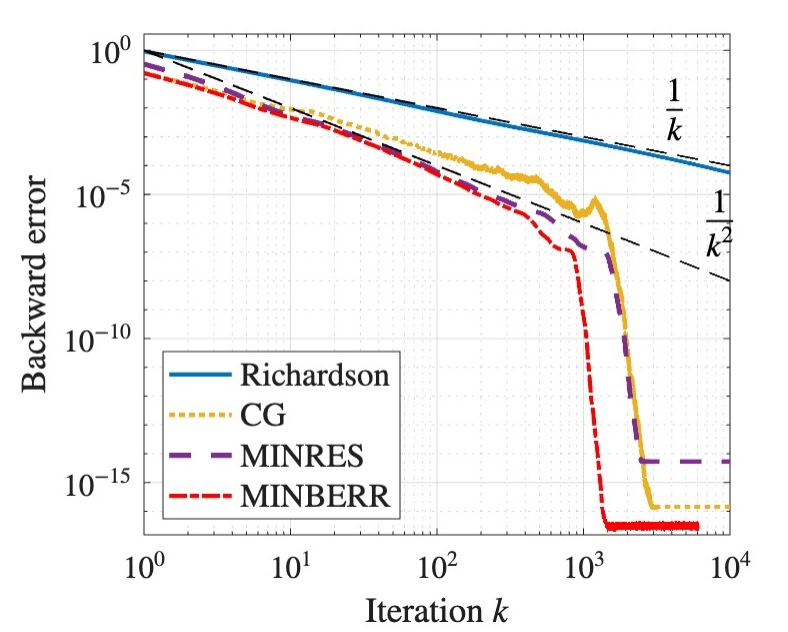}
     \label{fig:poisson3da}     \vspace{-3mm}\caption{\footnotesize\texttt{Pres\_Poisson} data:\\ PSD matrix, $n=14822$\\ (Fluid Dynamics)
    }
 \end{subfigure}
\begin{subfigure}{0.32\textwidth}
\includegraphics[width=\textwidth]{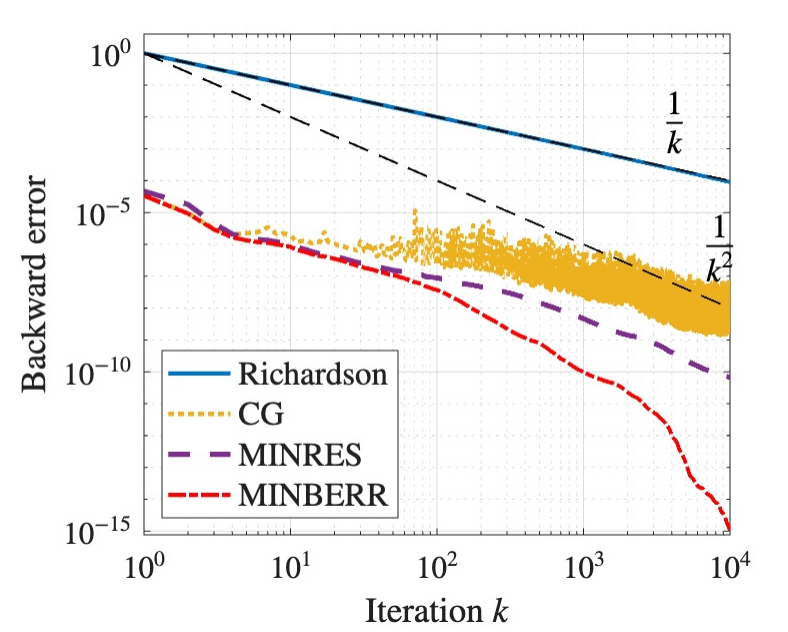}    \label{fig:olafu}
     \vspace{-3mm}\caption{\footnotesize\texttt{olafu} data:\\
PSD matrix, $n=16146$\\
(Aerospace Structural Model)
     } 
 \end{subfigure}
 \begin{subfigure}{0.32\textwidth}
\includegraphics[width=\textwidth]{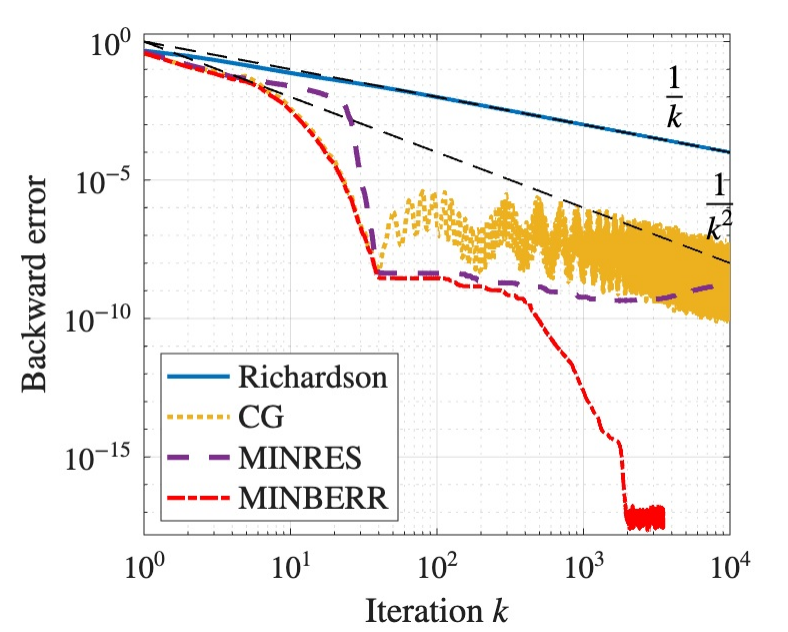}    \label{fig:raefsky}
     \vspace{-3mm}\caption{\footnotesize\texttt{raefsky4} data:\\
     PSD matrix, $n = 19779$\\
     (Structural Stability Analysis)
     } 
 \end{subfigure}

 \begin{subfigure}{0.32\textwidth}
     \includegraphics[width=\textwidth]{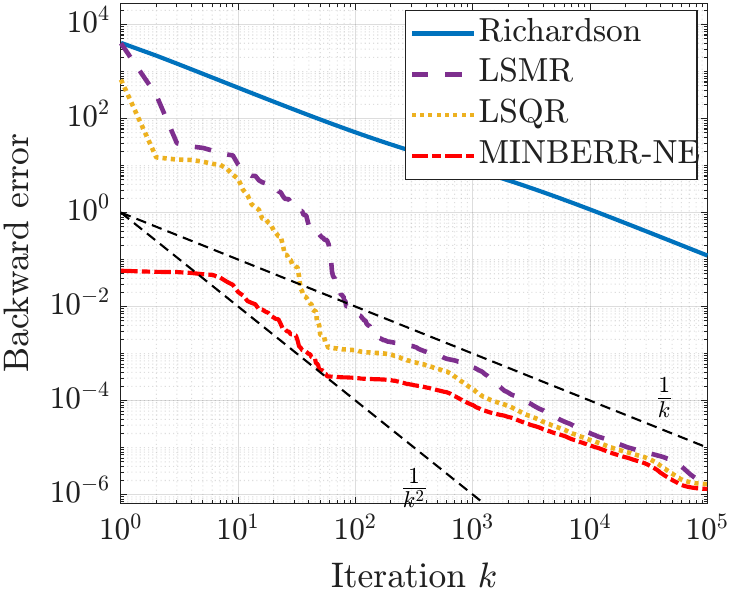}
     \label{fig:sherman}
\vspace{-3mm}\caption{\footnotesize\texttt{sherman3} data:\\
     square 
     non-PSD, $n = 5005$\\
     (Fluid Dynamics)} 
 \end{subfigure}
 \begin{subfigure}{0.32\textwidth}
\includegraphics[width=\textwidth]{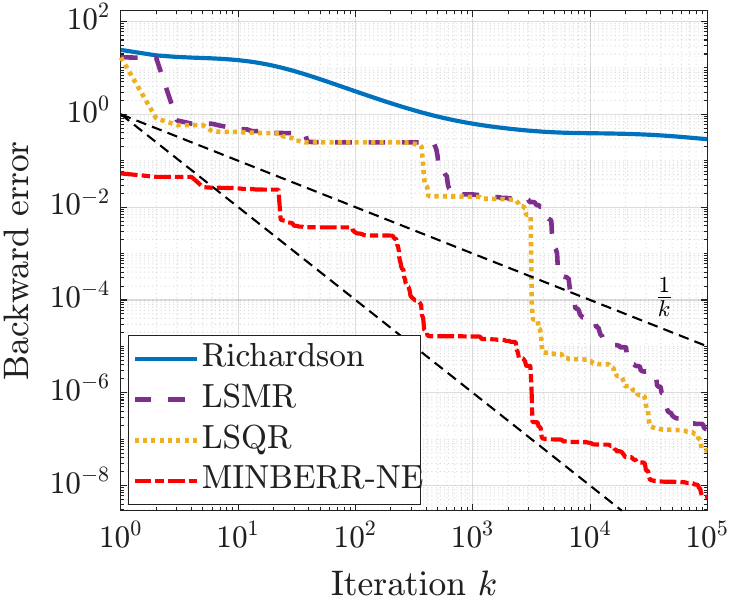}    \label{fig:bayer}
\vspace{-3mm}\caption{\footnotesize\texttt{bayer03} data:\\
     square non-PSD, $n = 6747$\\
     (Chemical Simulation)} 
 \end{subfigure}
 \begin{subfigure}{0.32\textwidth}
\includegraphics[width=\textwidth]{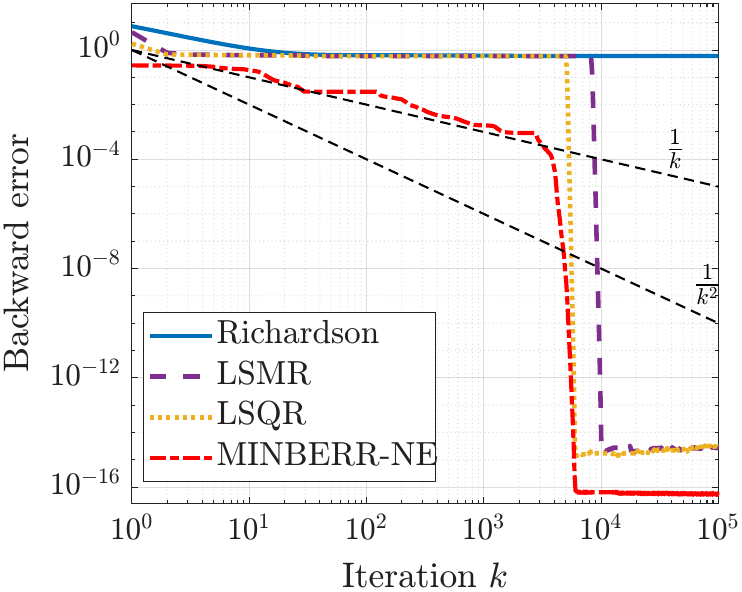}    \label{fig:cyl}
\vspace{-3mm}\caption{\footnotesize\texttt{cyl6} data:\\
    square non-PSD,
    $n = 13681$ \\
    (Shell Mechanics Simulation)
    } 
 \end{subfigure}

\vspace{-2mm} \caption{Empirical evaluation of \algpsd\ (a,b,c) and \algne\ (d,e,f) on benchmark problems \cite{davis2011university} with PSD and square non-PSD matrices, respectively.}
 \label{fig:real}
 \end{center}
\end{figure}

\section{Numerical Experiments}\label{s:experiments}

Here, we provide numerical experiments\footnote{The MATLAB code is publicly available at {\tt https://github.com/nakatsukasayuji/MINBERR}.}
 supporting our analysis, comparing MINBERR and \algne\ to popular methods on several test problems. We first illustrate the typical backward error convergence behavior of different methods on real-world tasks, both with PSD and general matrices. We then perform a more in-depth investigation of specific phenomena by evaluating individual methods on carefully designed synthetic hard problems.

\subsection{Benchmark Problems from Real Data}
We compare the performance of our proposed algorithms against that of several baselines on six problems from the SuiteSparse Matrix Collection \cite{davis2011university}.

\paragraph{PSD problems} In Figure \ref{fig:real} (a,b,c), we compare \algpsd\ against Richardson, CG, and MINRES on three PSD problems. For each method, we plot the backward error against iterations, and we also include lines for the two universal rates attained by our theory: $1/k$ (proven for Richardson) and $1/k^2$ (proven for \algpsd\ up to a constant factor). We observe that Richardson closely follows the theoretical rate, suggesting that the bound in Theorem~\ref{t:richardson} is often tight for typical ill-conditioned problems. On the other hand, \algpsd\ often outperforms its theoretical $1/k^2$ rate, especially later in the convergence. The behavior of CG and MINRES varies substantially between the problems, but they consistently outperform Richardson, and are consistently outperformed by \algpsd.

\paragraph{Non-Symmetric Problems} Next, in Figure \ref{fig:real} (d,e,f) we compare \algne\ against Richardson-NE (Richardson on normal equations), as well as LSQR and LSMR, on three general linear system tasks. First, we observe that all three baselines exhibit some initial blow-up of backward error, in that the convergence curves start far above 1. This phenomenon aligns with our analysis of Richardson on normal equations (Theorem \ref{t:failure}), and is further demonstrated below on synthetic data. \algne\ does not exhibit any initial blow-up and tends to converge at least at a $1/k$ rate. We do not observe the logarithmic dependence on the condition number from Theorem \ref{t:minberrne} on any of these benchmark problems.

\begin{figure}[t]
\includegraphics[width=1\textwidth]{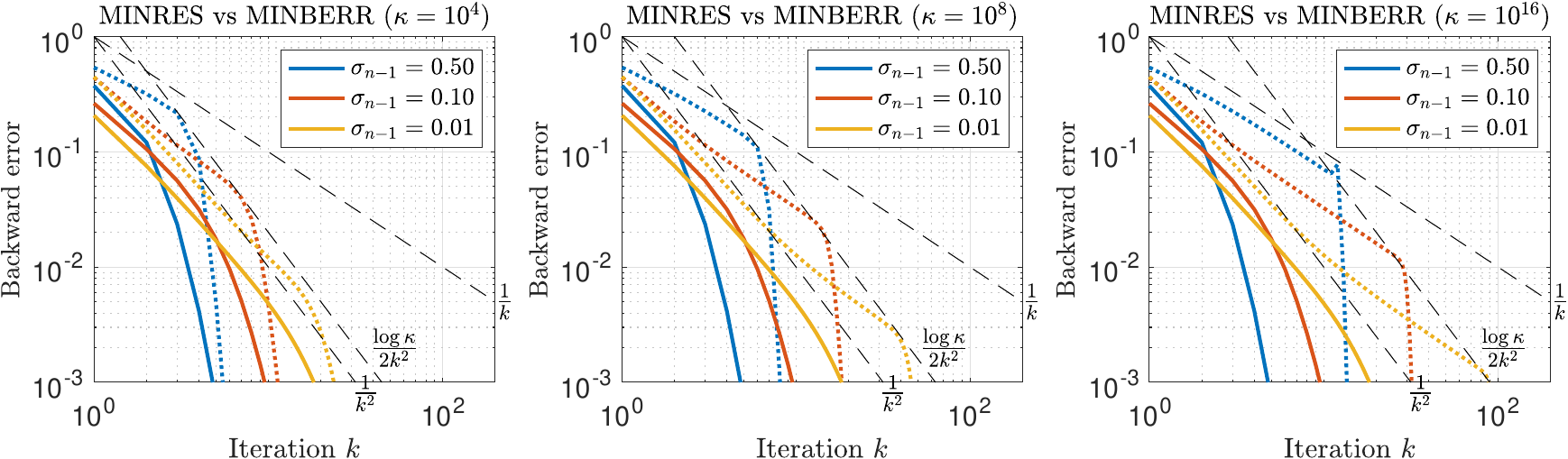}
\caption{Comparing MINRES with MINBERR on the \texttt{Small-Outlier}$(2000,\kappa,\sigma_{n-1})$ synthetic task, while varying $\kappa$ and $\sigma_{n-1}$. MINBERR is denoted by solid lines, while MINRES is shown using dotted lines.}
 \label{fig:minres-worstcase}
\end{figure}

\subsection{Hard Problems from Synthetic Matrices}
We next evaluate the algorithms on synthetic problems which are designed to induce their worst-case performance, and further illustrate our theoretical analysis. We define the following classes of problem instances:
\begin{enumerate}
    \item \texttt{Ill-Conditioned}$(n,\kappa)$. Let $\A$ be an $n\times n$ diagonal matrix with logarithmically spaced entries between 1 and $1/\kappa$, and let $\b = [1,...,1,\kappa]^\top\in\R^n$ consist of all 1's except for the last entry which is equal $\kappa$.
    \item 
    \texttt{Small-Outlier}$(n,\kappa,\sigma_{n-1})$. Let $\A$ be an $n\times n$ diagonal matrix with first $n-1$ entries logarithmically spaced between 1 and $\sigma_{n-1}$, and the last entry equal $1/\kappa$. Also, we let $\b = [1,...,1,\sqrt n]^\top\in\R^n$.
\end{enumerate}
We can extend these constructions to non-diagonal PSD matrices by substituting $\A\rightarrow \U\A\U^\top$, or to non-symmetric matrices with $\A\rightarrow \U\A\V^\top$, and replacing $\b$ with $\U\b$, for some orthogonal matrices $\U,\V$. Our PSD/non-symmetric experiments are invariant under those transformations (up to numerical precision).
In all of the following experiments we use problem dimension $n=2000$.

\paragraph{MINRES vs MINBERR}
We first demonstrate that popular Krylov solvers on PSD systems, even though often effective in practice, may exhibit a backward error rate no better than Richardson in the worst case. Here, we illustrate this by running MINRES on the \texttt{Small-Outlier} problem while varying $\kappa$ and $\sigma_{n-1}$. In Figure \ref{fig:minres-worstcase}, we observe that by adjusting these problem parameters we can force MINRES (dotted lines) to converge nearly as slowly as $1/k$ for an extended number of iterations. Specifically, the results suggest that for any sufficiently large $k$ and $\kappa$, we can find a problem instance for which the backward error of MINRES after $k$ iterations is observed to be at least $\min\{1/k, \log_{10}(\kappa)/(2k^2)\}$. By increasing $\kappa$, we can force the $1/k$ rate to dominate, although this is limited by the mild logarithmic dependence on $\kappa$ observed for the MINRES backward error. We note that MINBERR (solid lines) is not affected by this phenomenon, avoiding condition number dependence altogether. 

\begin{figure}[t]
\includegraphics[width=\textwidth]{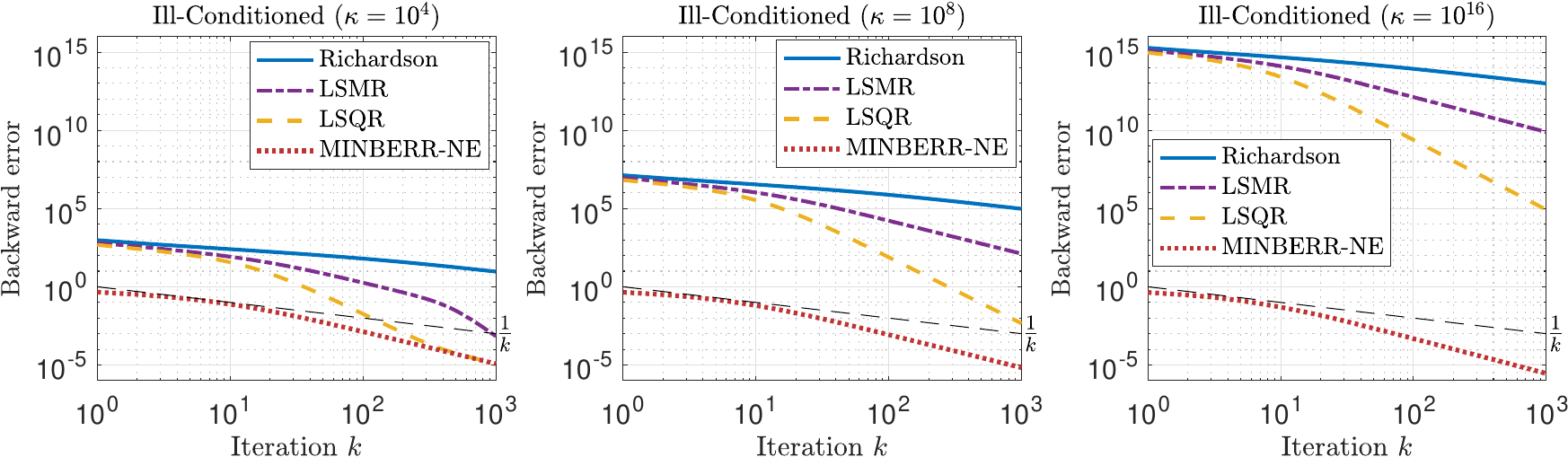}
\caption{Comparing \algne\ against Richardson-NE, LSMR, and LSQR on the \texttt{Ill-Conditioned}$(2000,\kappa)$ synthetic task treated as a general non-symmetric system.}
 \label{fig:nonsym}
\end{figure}

\paragraph{Error Blow-up on Normal Equations}
Turning to non-symmetric systems, we next demonstrate that the initial backward error for Richardson-NE/LSQR/LSMR can indeed be as large as the condition number of the matrix, as suggested by Theorem~\ref{t:failure}. In Figure \ref{fig:nonsym}, we illustrate this  on the \texttt{Ill-Conditioned} problem with a range of $\kappa$ values. We observe that the backward error blow-up is nearly as severe for LSQR and LSMR as it is for Richardson, although the Krylov solvers exhibit faster convergence after that. Finally, similarly as in all of the benchmark problems, \algne\ does not exhibit any backward error blow-up and its convergence curves consistently lie below the $1/k$ threshold.

\paragraph{Stagnation of \algne}
Unfortunately, it turns out that the logarithmic dependence of \algne\ on the condition number predicted by Theorem \ref{t:minberrne} can be observed in practice. We show this empirically in Figure \ref{fig:nonsym_hard} (top plots), again using the \texttt{Small-Outlier} family of problems. Here, by adjusting the parameters we can force the backward error of \algne\ to stagnate close to 1 for an extended number of iterations. Specifically, our experiments suggest that for any sufficiently large $k$ and $\kappa$, we can find a problem instance for which the backward error of \algne\ after $k$ iterations is at least $\min\{1, \log_{10}(\kappa)/k\}$.

As we can see in Figure \ref{fig:nonsym_hard} (dotted lines), the \algne\ backward error exhibits an initial $1/k$ convergence until reaching a plateau around the value of $\sigma_{n-1}$, and then eventually it resumes fast convergence. By taking the outlier singular value to zero (blowing up $\kappa$), we can extend the plateau, making backward error stagnate for as long as we want (up to the numerical limitations that arise when increasing the condition number). We note that this phenomenon also occurs if we replace the single small outlier with a large cluster of small singular values, or if we make the problem non-symmetric by applying different orthogonal matrices $\U$ and $\V$ on each side.

\begin{figure}[t]
\includegraphics[width=\textwidth]{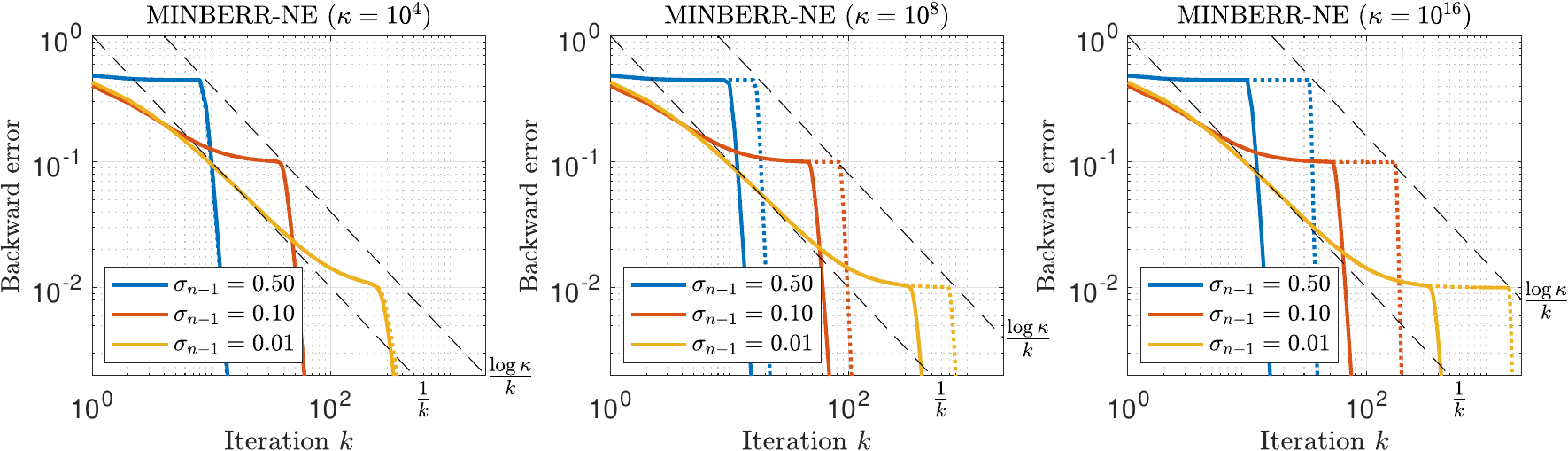}
 \caption{Evaluating \algne\ on the \texttt{Small-Outlier}$(2000,\kappa,\sigma_{n-1})$ synthetic task treated as a general non-symmetric linear system. We use dotted lines to denote \algne\ running on the original task, while solid lines denote \algne\ running on the matrix perturbed with Gaussian noise (in both cases, backward error is computed with respect to the original task).}
 \label{fig:nonsym_hard}
\end{figure}

\paragraph{Perturbed \algne}
As shown in Theorem \ref{t:perturbed-minberrne}, we can get around this issue by perturbing the matrix to reduce its condition number. 
In Figure \ref{fig:nonsym_hard} (solid lines), we replace $\A$ with $\tilde\A=\A + \epsilon\frac{\|\A\|_2}{\|\G\|_2}\G$, where $\G$ has independent Gaussian entries and we let $\epsilon=10^{-3}$. We then run \algne\ on $\tilde\A$ and $\b$, while still evaluating its backward error with respect to the original matrix $\A$. 
We observe that perturbed \algne\ is insensitive to increasing the condition number of the problem.

\section{Conclusions and Future Directions}
\label{s:conclusions}
We showed that, when measured in terms of backward error, the convergence of iterative linear system solvers can be made universal, i.e., independent of the conditioning of the problem. For PSD linear systems, we gave a universal $1/k$ backward error convergence guarantee for the classical Richardson iteration. We then developed an efficient Krylov subspace method, MINBERR, which attains an even better $O(1/k^2)$ universal rate and has strong empirical performance when compared against CG and MINRES on benchmark~problems. 

For general linear systems, we extended our algorithm via the normal equations obtaining \algne\ which attains $O(\log(\kappa(\A))/k)$ backward error. A small Gaussian perturbation of the input ensures a nearly universal $O(\log(n)/k)$ convergence rate of the method. This raises the question whether a dense random perturbation, or indeed any source of randomness, is necessary to achieve this guarantee, and also whether truly universal convergence (i.e., independent of $n$) can be attained for general linear systems. 
Another important future direction is to provide a finite-precision backward error analysis of MINBERR and \algne\ in order to verify the numerical stability of these methods. Finally, despite MINBERR's universal convergence, preconditioning remains a powerful technique to accelerate linear system solvers. Developing a preconditioned version of MINBERR is left for future work.

\section*{Acknowledgments} The authors thank Anil Damle, Ethan Epperly,  Raphael Meyer, Cameron Musco, and Christopher Musco for helpful conversations. An LLM-based assistant was used during the development of some of the proofs. The authors verified all steps and take responsibility for the results.

\bibliographystyle{siamplain}
\bibliography{bib2}

\def\noopsort#1{}\def\l{\char32l}\def\v#1{{\accent20 #1}} \let\^^_=\v\def\hbk{hardback}\def\pbk{paperback}
\begin{thebibliography}{10}

\bibitem{allen2016optimal}
{\sc Z.~Allen-Zhu and E.~Hazan}, {\em Optimal black-box reductions between optimization objectives}, Advances in Neural Information Processing Systems, 29 (2016).

\bibitem{alman2025more}
{\sc J.~Alman, R.~Duan, V.~V. Williams, Y.~Xu, Z.~Xu, and R.~Zhou}, {\em More asymmetry yields faster matrix multiplication}, in Proceedings of the 2025 Annual ACM-SIAM Symposium on Discrete Algorithms (SODA), SIAM, 2025, pp.~2005--2039.

\bibitem{arioli1992stopping}
{\sc M.~Arioli, I.~Duff, and D.~Ruiz}, {\em Stopping criteria for iterative solvers}, SIAM J. Matrix Anal. Appl., 13 (1992), pp.~138--144.

\bibitem{axelsson2000sublinear}
{\sc O.~Axelsson and I.~Kaporin}, {\em On the sublinear and superlinear rate of convergence of conjugate gradient methods}, Numerical Algorithms, 25 (2000), pp.~1--22.

\bibitem{axelsson1986rate}
{\sc O.~Axelsson and G.~Lindskog}, {\em On the rate of convergence of the preconditioned conjugate gradient method}, Numerische Mathematik, 48 (1986), pp.~499--523.

\bibitem{barrett2020implicit}
{\sc D.~G. Barrett and B.~Dherin}, {\em Implicit gradient regularization}, arXiv preprint arXiv:2009.11162,  (2020).

\bibitem{beerens2024adversarial}
{\sc L.~Beerens and D.~J. Higham}, {\em Adversarial ink: Componentwise backward error attacks on deep learning}, IMA Journal of Applied Mathematics, 89 (2024), pp.~175--196.

\bibitem{bertsekas2016}
{\sc D.~P. Bertsekas}, {\em {Nonlinear} {Programming}}, Athena Scientific, 3rd~ed., 2016.

\bibitem{chou1987optimality}
{\sc A.~W. Chou}, {\em On the optimality of krylov information}, Journal of Complexity, 3 (1987), pp.~26--40.

\bibitem{davidson2001local}
{\sc K.~R. Davidson and S.~J. Szarek}, {\em Local operator theory, random matrices and banach spaces}, Handbook of the geometry of Banach spaces, 1 (2001), p.~131.

\bibitem{davis2011university}
{\sc T.~A. Davis and Y.~Hu}, {\em The {U}niversity of {F}lorida sparse matrix collection}, ACM Trans. Math. Soft., 38 (2011), pp.~1--25.

\bibitem{derezinski2026matrix}
{\sc M.~Derezi{\'n}ski, E.~N. Epperly, and R.~A. Meyer}, {\em The matrix-vector complexity of {$Ax=b$}}, arXiv preprint arXiv:2602.04842,  (2026).

\bibitem{di2023backward}
{\sc S.~Di~Giovacchino, D.~J. Higham, and K.~Zygalakis}, {\em Backward error analysis and the qualitative behaviour of stochastic optimization algorithms: Application to stochastic coordinate descent}, arXiv preprint arXiv:2309.02082,  (2023).

\bibitem{d2021acceleration}
{\sc A.~d’Aspremont, D.~Scieur, and A.~Taylor}, {\em Acceleration methods}, Foundations and Trends in Optimization, 5 (2021), pp.~1--245.

\bibitem{epperly2026fast}
{\sc E.~N. Epperly, M.~Meier, and Y.~Nakatsukasa}, {\em Fast randomized least-squares solvers can be just as accurate and stable as classical direct solvers}, Communications on Pure and Applied Mathematics, 79 (2026), pp.~293--339.

\bibitem{feng2019uniform}
{\sc Y.~Feng, T.~Gao, L.~Li, J.-G. Liu, and Y.~Lu}, {\em Uniform-in-time weak error analysis for stochastic gradient descent algorithms via diffusion approximation}, arXiv preprint arXiv:1902.00635,  (2019).

\bibitem{dqdsori}
{\sc K.~V. Fernando and B.~N. Parlett}, {\em Accurate singular values and differential qd-algorithms}, Numer. Math., {67} ({1994}), pp.~{191--229}.

\bibitem{fong2011lsmr}
{\sc D.~C.-L. Fong and M.~Saunders}, {\em Lsmr: An iterative algorithm for sparse least-squares problems}, SIAM Journal on Scientific Computing, 33 (2011), pp.~2950--2971.

\bibitem{golubbook}
{\sc G.~H. Golub and C.~F. Van~Loan}, {\em Matrix Computations}, {The Johns Hopkins University Press}, 1996.

\bibitem{golub2013matrix}
{\sc G.~H. Golub and C.~F. Van~Loan}, {\em Matrix computations}, JHU press, 2013.

\bibitem{greenbaum1989behavior}
{\sc A.~Greenbaum}, {\em Behavior of slightly perturbed lanczos and conjugate-gradient recurrences}, Linear Algebra and its Applications, 113 (1989), pp.~7--63.

\bibitem{greenbaumbook}
{\sc A.~Greenbaum}, {\em Iterative Methods for Solving Linear Systems}, SIAM, Philadelphia, PA, USA, 1997.

\bibitem{greenbaum1996any}
{\sc A.~Greenbaum, V.~Pt{\'a}k, and Z.~e.~k. Strako{\v{s}}}, {\em Any nonincreasing convergence curve is possible for {GMRES}}, SIAM J. Matrix Anal. Appl., 17 (1996), pp.~465--469.

\bibitem{hestenes1952methods}
{\sc M.~R. Hestenes, E.~Stiefel, et~al.}, {\em Methods of conjugate gradients for solving linear systems}, Journal of research of the National Bureau of Standards, 49 (1952), pp.~409--436.

\bibitem{Higham:2002:ASNA}
{\sc N.~J. Higham}, {\em Accuracy and Stability of Numerical Algorithms}, SIAM, Philadelphia, PA, USA, second~ed., 2002.

\bibitem{higham2022mixed}
{\sc N.~J. Higham and T.~Mary}, {\em Mixed precision algorithms in numerical linear algebra}, Acta Numerica, 31 (2022), pp.~347--414.

\bibitem{kasenally1995gmback}
{\sc E.~M. Kasenally}, {\em {GMBACK}: a generalised minimum backward error algorithm for nonsymmetric linear systems}, SIAM J. Sci. Comput., 16 (1995), pp.~698--719.

\bibitem{kasenally1997analysis}
{\sc E.~M. Kasenally and V.~Simoncini}, {\em Analysis of a minimum perturbation algorithm for nonsymmetric linear systems}, SIAM J. Numer. Anal., 34 (1997), pp.~48--66.

\bibitem{kuczynski1992estimating}
{\sc J.~Kuczy{\'n}ski and H.~Wo{\'z}niakowski}, {\em Estimating the largest eigenvalue by the power and lanczos algorithms with a random start}, SIAM J. Matrix Anal. Appl., 13 (1992), pp.~1094--1122.

\bibitem{laurent2000adaptive}
{\sc B.~Laurent and P.~Massart}, {\em Adaptive estimation of a quadratic functional by model selection}, Ann. Stat.,  (2000), pp.~1302--1338.

\bibitem{lee2019first}
{\sc C.-p. Lee and S.~Wright}, {\em First-order algorithms converge faster than $ o (1/k) $ on convex problems}, in International Conference on Machine Learning, PMLR, 2019, pp.~3754--3762.

\bibitem{levin1988degree}
{\sc A.~Levin and E.~Saff}, {\em Degree of approximation of real functions by reciprocals of real and complex polynomials}, SIAM journal on mathematical analysis, 19 (1988), pp.~233--245.

\bibitem{musco2018stability}
{\sc C.~Musco, C.~Musco, and A.~Sidford}, {\em Stability of the lanczos method for matrix function approximation}, in Proceedings of the Twenty-Ninth Annual ACM-SIAM Symposium on Discrete Algorithms, SIAM, 2018, pp.~1605--1624.

\bibitem{nemirovskij1983problem}
{\sc A.~S. Nemirovsky and D.~B. Yudin}, {\em Problem Complexity and Method Efficiency in Optimization}, Wiley-Interscience, 1983, \url{https://www2.isye.gatech.edu/~nemirovs/Nemirovskii_Yudin_1983.pdf}.

\bibitem{nesterov1983method}
{\sc Y.~Nesterov}, {\em A method for solving the convex programming problem with convergence rate o (1/k2)}, in Dokl akad nauk Sssr, vol.~269, 1983, p.~543.

\bibitem{paige1976error}
{\sc C.~C. Paige}, {\em Error analysis of the lanczos algorithm for tridiagonalizing a symmetric matrix}, IMA Journal of Applied Mathematics, 18 (1976), pp.~341--349.

\bibitem{paige1982lsqr}
{\sc C.~C. Paige and M.~A. Saunders}, {\em {LSQR}: An algorithm for sparse linear equations and sparse least squares}, ACM Trans. Math. Soft., 8 (1982), pp.~43--71.

\bibitem{paige2002residual}
{\sc C.~C. Paige and Z.~Strakos}, {\em Residual and backward error bounds in minimum residual krylov subspace methods}, SIAM J. Sci. Comput., 23 (2002), pp.~1898--1923.

\bibitem{peters1979inverse}
{\sc G.~Peters and J.~H. Wilkinson}, {\em Inverse iteration, ill-conditioned equations and {N}ewton's method}, SIAM Rev., 21 (1979), pp.~339--360.

\bibitem{richardson1911ix}
{\sc L.~F. Richardson}, {\em Ix. the approximate arithmetical solution by finite differences of physical problems involving differential equations, with an application to the stresses in a masonry dam}, Philosophical Transactions of the Royal Society of London. Series A, containing papers of a mathematical or physical character, 210 (1911), pp.~307--357.

\bibitem{saad2003iterative}
{\sc Y.~Saad}, {\em Iterative methods for sparse linear systems}, SIAM, 2003.

\bibitem{sankar2006smoothed}
{\sc A.~Sankar, D.~A. Spielman, and S.-H. Teng}, {\em Smoothed analysis of the condition numbers and growth factors of matrices}, SIAM J. Matrix Anal. Appl., 28 (2006), pp.~446--476.

\bibitem{trefbau}
{\sc L.~N. Trefethen and D.~Bau}, {\em Numerical Linear Algebra}, SIAM, Philadelphia, 1997.

\bibitem{virtanen2020scipy}
{\sc P.~Virtanen, R.~Gommers, T.~E. Oliphant, M.~Haberland, T.~Reddy, D.~Cournapeau, E.~Burovski, P.~Peterson, W.~Weckesser, J.~Bright, et~al.}, {\em Scipy 1.0: fundamental algorithms for scientific computing in python}, Nature methods, 17 (2020), pp.~261--272.

\bibitem{von1947numerical}
{\sc J.~von Neumann and H.~Goldstine}, {\em Numerical inverting of matrices of high order}, Bulletin of the American Mathematical Society, 53 (1947), pp.~1021--1099.

\bibitem{wilkinson1960error}
{\sc J.~H. Wilkinson}, {\em Error analysis of floating-point computation}, Numerische Mathematik, 2 (1960), pp.~319--340.

\bibitem{wilkinson1961error}
{\sc J.~H. Wilkinson}, {\em Error analysis of direct methods of matrix inversion}, Journal of the ACM (JACM), 8 (1961), pp.~281--330.

\bibitem{wilkinson1963rounding}
{\sc J.~H. Wilkinson}, {\em Rounding errors in algebraic processes}, Her Majesty's Stationery Office, 1963.

\bibitem{wilkinson:1965}
{\sc J.~H. Wilkinson}, {\em The Algebraic Eigenvalue Problem}, Oxford University Press, 1965.

\end{thebibliography}

\appendix 

\section{Analysis of inverse iteration: Proof of Lemma~\ref{lemma:invit}}
\label{a:interse-iteration}

Let $\lambda_{\min}=\lambda_1\leq \lambda_2\cdots \leq \lambda_n$ be the eigenvalues of $\M$, and let ${\bf u}_i$ be the corresponding eigenvectors. 
Expressing ${\bf v}_0 = \sum_{i=1}^n c_i {\bf u}_i, $ where by the assumption that ${\bf v}_0$ is Gaussian, $c_i \sim \mathcal{N}(0, 1)$, we then have 
$\M^{-{\ell}}{\bf v}_0 = \sum_{i=1}^n c_i \lambda_i^{-{\ell}} {\bf u}_i$.
Our goal is to find ${\ell}$ such that the Rayleigh quotient 
$R_{\M}({\bf v}_{\ell}):={\bf v}_{\ell}^\top \M{\bf v}_{\ell}/\|{\bf v}_{\ell}\|_2^2 \leq (1+\Delta)\lambda_1$ with probability $1-\delta$, independent of the specific values of $\lambda_i$.

We partition the index set $I = \{1, \dots, n\}$ into two based on the target error $\Delta$:
\begin{align*} 
 S_{\mathrm{low}} = \{i : \lambda_i \leq (1+\Delta/2)\lambda_1\}
    \quad\text{and}\quad S_{\mathrm{high}} = \{i : \lambda_i > (1+\Delta/2)\lambda_1\}.
\end{align*}
Note that index $1$ is always in $S_{\mathrm{low}}$, and the Rayleigh quotient is:
\begin{equation}\nonumber
    R_{\M}({\bf v}_{\ell}) = \frac{\|\M^{-{\ell}+1/2}{\bf v}_0\|_2^2}{\|\M^{-{\ell}}{\bf v}_0\|_2^2} =\frac{\sum_{i \in S_{\mathrm{low}}} c_i^2 \lambda_i^{-(2{\ell}-1)} + \sum_{j \in S_{\mathrm{high}}} c_j^2 \lambda_j^{-(2{\ell}-1)}}{\sum_{i \in S_{\mathrm{low}}} c_i^2 \lambda_i^{-2{\ell}} + \sum_{j \in S_{\mathrm{high}}} c_j^2 \lambda_j^{-2{\ell}}}.
\end{equation}

We seek to bound the relative error $\frac{R_{\M}({\bf v}_{\ell}) - \lambda_1}{\lambda_1}$. Since $\lambda_i \leq (1+\Delta/2)\lambda_1$ for all $i \in S_{\mathrm{low}}$, the contribution of the first sum in the numerator is already ``safe'', and the error is driven by $S_{\mathrm{high}}$. We simplify the quotient by considering only the $c_1$ term in the denominator:
\begin{align*}
    R_{\M}({\bf v}_{\ell}) &\leq \frac{(1+\Delta/2)\lambda_1 \sum_{i \in S_{\mathrm{low}}} c_i^2 \lambda_i^{-2{\ell}} + \sum_{j \in S_{\mathrm{high}}} c_j^2 \lambda_j^{-(2{\ell}-1)}}{\sum_{i \in S_{\mathrm{low}}} c_i^2 \lambda_i^{-2{\ell}}}
 \\
 &\leq (1+\Delta/2)\lambda_1 + \frac{\sum_{j \in S_{\mathrm{high}}} c_j^2 \lambda_j^{-(2{\ell}-1)}}{c_1^2 \lambda_1^{-2{\ell}}}.
\end{align*}
Focusing on the second term, for $j \in S_{\mathrm{high}}$, $\lambda_j > (1+\Delta/2)\lambda_1$. Thus:
\begin{equation}\nonumber
    \frac{\lambda_j^{-(2{\ell}-1)}}{\lambda_1^{-2{\ell}}} = \lambda_1 \left( \frac{\lambda_1}{\lambda_j} \right)^{2{\ell}-1} < \lambda_1 \left( \frac{1}{1+\Delta/2} \right)^{2{\ell}-1}
\end{equation}
Altogether, the relative error is bounded as follows:
\begin{equation}\nonumber
    \frac{R_{\M}({\bf v}_{\ell}) - \lambda_1}{\lambda_1}\leq \frac{\Delta}{2}  + \frac{\sum_{j \in S_{\mathrm{high}}} c_j^2}{c_1^2} \left( \frac{1}{1+\Delta/2} \right)^{2{\ell}-1}=:\frac{\Delta}{2} +\mbox{Noise}.
\end{equation}

Now, using properties of $\chi^2$ distributions we observe the following:
\begin{enumerate}
    \item $\sum c_j^2$ follows $\chi^2_{|S_{\mathrm{high}}|}$, so by~\cite[Eqn.~(4.3)]{laurent2000adaptive} we see that with probability $1-\delta/2$, we have $\sum c_j^2 \leq n \ln(2/\delta)$.
    \item $c_1^2$ follows $\chi^2_1$, so since the pdf of a Gaussian random variable is bounded by $1/\sqrt{2\pi}$, we have $\mathbb{P}(c_1^2 < \eta) \leq \sqrt{2\eta/\pi}$. To ensure that $c_1^2 \geq \eta$ with probability $1-\delta/2$, we need $\eta \leq \frac{\pi \delta^2}{8}$; we thus set $\eta=\frac{\pi \delta^2}{8}$. 
\end{enumerate}
Combining these, with probability at least $1 - \delta$:
\begin{equation}
    \frac{\sum_{j \in S_{\mathrm{high}}} c_j^2}{c_1^2} \leq \frac{16n \ln(2/\delta)}{\pi \delta^2}.
\end{equation}
For any $j \in S_{\mathrm{high}}$, we have $\lambda_j > (1 + \frac{\Delta}{2})\lambda_1$. The noise term is bounded by:
\begin{equation}
    \text{Noise} \leq \left( \frac{\sum_{j \in S_{\mathrm{high}}} c_j^2}{c_1^2} \right) \left( \frac{\lambda_1}{\lambda_j} \right)^{2\ell-1} \leq \left( \frac{16n \ln(2/\delta)}{\pi \delta^2} \right) \left( \frac{1}{1 + \Delta/2} \right)^{2\ell-1}
\end{equation}
To achieve $R_\M({\bf v}_\ell) \leq (1 + \Delta)\lambda_1$, we require the noise term to be $\leq\Delta/2$:
\begin{equation}
    \left( \frac{1}{1 + \Delta/2} \right)^{2\ell-1} \leq \frac{\pi \delta^2 \Delta}{32n \ln(2/\delta)}
\end{equation}
Taking the logarithm and solving for $\ell$ completes the proof:
\begin{equation}
    \ell \geq \frac{1}{2} \left[ 1 + \frac{\ln\left( \frac{32n \ln(2/\delta)}{\pi \delta^2 \Delta} \right)}{\ln(1 + \Delta/2)} \right]. 
\end{equation}

In particular, for a $1/2$-relative error ($\Delta = 0.5$), we require $\ell \geq 2.23 \ln(n/\delta^2)$. 

\section{Proof of Proposition~\ref{prop:minberrNE}}
\label{a:minberr-ne} 
We start by examining~\eqref{eq:bAQ}, in which we note that $[\b, \A\Q_k] = \U_k[\|\b\|_2\e,\B_k]=:\U_k{\bf N}$, where ${\bf N}\in\mathbb{R}^{(k+1)\times (k+1)}$. We then see that~\eqref{eq:bAQ} reduces to
\begin{equation}\label{eq:bAQ2}    
\bf{N}^\top 
\bf{N}\begin{bmatrix}
-1 \\\y     
\end{bmatrix}
=\lambda\begin{bmatrix}
0 & \\ & \I_k
\end{bmatrix}^\top 
\begin{bmatrix}
0 & \\ & \I_k
\end{bmatrix}
\begin{bmatrix}
-1 \\\y     
\end{bmatrix}. 
\end{equation}
Now $\bf{N}$ is upper bidiagonal; we can apply a congruence transformation to zero-out the $(1,2)$  entry, that is, defining the matrix 
$\W = 
\begin{bmatrix}
\multicolumn{2}{c}{\w_1} \\     
0 & \I_{k}
\end{bmatrix}
\in\mathbb{R}^{(k+1)\times (k+1)}$
where 
$\w_1 = [1, -\frac{{\bf N}_{1,2}}{{\bf N}_{1,1}},0,\ldots,0]$ (where we note that $\frac{{\bf N}_{1,2}}{{\bf N}_{1,1}}=\frac{(\B_k)_{1,1}}{\|\b\|_2}$),
that is, $\W$ is equal to $\I_{k+1}$ except in the (1,2) entry. 
Then consider 
\begin{equation}\nonumber 
\W^\top {\bf N}^\top 
{\bf N}
\W
\left(\W^{-1}
\begin{bmatrix}
-1 \\\y     
\end{bmatrix}\right)
=\lambda\W^\top\begin{bmatrix}
0 & \\ & \I_k
\end{bmatrix}^\top 
\begin{bmatrix}
0 & \\ & \I_k
\end{bmatrix}\W 
\left(\W^{-1}
\begin{bmatrix}
-1 \\\y     
\end{bmatrix}\right), 
\end{equation}
which simplifies to 
\begin{equation}\label{eq:tosvd}
\begin{bmatrix}
\|\b\|_2& \\
& \hspace{-2mm}\tilde\B_k
\end{bmatrix}^\top
\begin{bmatrix}
\|\b\|_2& \\
& \hspace{-2mm}\tilde\B_k
\end{bmatrix}
\left(\W^{-1}
\begin{bmatrix}
-1 \\\y     
\end{bmatrix}\right)
=\lambda\begin{bmatrix}
0 & \\ & \I_k
\end{bmatrix}^\top 
\begin{bmatrix}
0 & \\ & \I_k
\end{bmatrix}
\left(\W^{-1}
\begin{bmatrix}
-1 \\\y     
\end{bmatrix}\right)
\end{equation}
where $\tilde\B_k$ is the submatrix of $\B$ obtained by removing the first row. 

As in the PSD case, we recognize~\eqref{eq:tosvd} 
as a
block diagonal generalized eigenvalue problem (whose smallest eigenpair is desired)
with the first $1\times 1$ block having eigenvalue at infinity. The second block has $\I_k$ in the right-hand side, so the problem is equivalent to finding the smallest singular value and its corresponding right singular vector of $\tilde\B_k$, which is bidiagonal. 
The proof of $\berr_{\A,\b}(\widehat\x) = \|\tilde \B_k\widehat{\bf v}\|_2/\|\widehat {\bf v}\|_2$ is as in Lemma~\ref{l:quotient}, noting that $(\widehat\x)_1=-1$.

\section{Polynomial approximation}\label{app:cheb}
 We use the following simple fact.
\begin{lemma}\label{lemma:sin-eq-easy}
    If $\gamma \in [0, \pi/2]$ then for any integer $a \ge 1$, $\sin(a\gamma) \le a \sin(\gamma)$.
\end{lemma}
\begin{proof}
    This holds by induction on $a$. When $a = 1$, the result trivially holds, assume that it holds for some $a > 1$. Then
    \begin{align*}
    \sin((a+1)\gamma) &= \sin (a \gamma) \cos(\gamma) + \cos(a\gamma) \sin(\gamma) \\
    &\le a\sin(\gamma)\cdot 1 + \sin(\gamma) \cdot 1 = (a + 1) \sin(\gamma).
    \end{align*}
\end{proof}

\begin{proof}[Proof of Lemma~\ref{lem:cheb}]
Recall that
\begin{align*}
G(x) = \frac{1 - T^*_\ell(x)}{2 \ell^2} = \frac{1 - \cos(\ell \arccos(2x - 1))}{2 \ell^2}.
\end{align*}
For $x = \sin^2(\gamma)$ we have $2x -1 = - \cos(2 \gamma)$ and $\arccos(-\cos(2\gamma)) = \pi - 2 \gamma.$ Since $\ell$ is even by construction, $\cos(\ell(\pi - 2\gamma)) = \cos(2\ell \gamma) = 1 - 2 \sin^2(\ell \gamma).$
So,
$$
G(x) = \frac{1 - 1 + 2 \sin^2(\ell \gamma)}{2 \ell^2} = \frac{\sin^2(\ell \gamma)}{\ell^2}.
$$
Using Lemma~\ref{lemma:sin-eq-easy}, we have then
$$
G(x) \le \frac{\ell^2 \sin^2(\gamma)}{\ell^2} = x,
$$
and $G(x) \ge 0$ as $T_\ell^* \in [-1,1]$. Finally, zero is the root of $f(x) := x - G(x)$ multiplicity at least two because
$$
f(0) = - G(0) = \frac{1 - \cos(\ell \arccos(-1))}{\ell^2} = 0; \quad 
f'(0) = 1 - G'(0) = 1 - \frac{2 \ell^2 (-1)^\ell}{2\ell^2} = 0
$$
using that $\ell$ is even  and the standard derivative properties of Chebyshev polynomials.
\end{proof}

\begin{proof}[Proof of Lemma~\ref{lem:sharp-trig-bound}]
First, we show that $F_\ell(\gamma)$ cannot be maximized at $\gamma \in (\frac{\pi}{2\ell}, \frac{\pi}{2}]$. Indeed,
let 
\[
v:=\frac{1}{\ell}\arcsin\bigl|\sin(\ell \gamma)\bigr|
\in \left[0,\frac{\pi}{2\ell}\right].
\]
Then, $\sin^2(\ell v)=\sin^2(\ell \gamma)$ but $v \le \gamma$, so $\sin^2 v\le \sin^2 \gamma$ since $\sin(\cdot)$ is increasing on $[0, \frac{\pi}{2}]$. Also, for \(y\in[0,1]\), the function
$
\phi_y(x)=\frac{xy}{x-y}
$
is decreasing, so $F_\ell(\gamma)\le F_\ell(v)$.

\medskip

Now, we show that $F_\ell(\gamma)$ is decreasing on $\gamma \in [0, \frac{\pi}{2\ell}]$. To simplify the computation, we show that $1/F_\ell$ has nonnegative derivative on this interval. Define
\[
G_\ell(u):=\frac{1}{F_\ell(u)}
=\csc^2(\ell u)-\frac{1}{\ell^2}\csc^2 u,
\]
\[
G_\ell'(u)=\frac{2}{\ell^2u^3}\bigl(h(u)-h(\ell u)\bigr),
\qquad
h(x):=x^3\csc^2 x\,\cot x
=\frac{x^3\cos x}{\sin^3 x}.
\]
So it suffices to know that \(h\) is decreasing on \((0,\pi/2]\). Consider
\[
h'(x)=\frac{x^2}{\sin^4 x}\Bigl(3\sin x\cos x-x(1+2\cos^2 x)\Bigr).
\]
Let $p(x):=3\sin x\cos x-x(1+2\cos^2 x).$ Then
$p'(x)=4\sin x\,(x\cos x-\sin x)\le 0,$
because \(x\mapsto \sin x/x\) decreases on \((0,\pi/2]\), or equivalently
\[
0\geq \Big(\frac{\sin x}{x}\Big)' = \frac{x \cos x - \sin x}{x^2}.
\]
Since \(p(0)=0\), we get \(p(x)\le 0\), hence
$h'(x)\le 0$, so \(h\) is decreasing. Therefore, for \(0<u\le \pi/(2\ell)\), since \(\ell u\le \pi/2\), we have
$h(u)\ge h(\ell u)$,
which implies that
$G_\ell'(u)\ge 0$.

\medskip

Therefore, the supremum of $F_\ell$ is the limit as \(u\downarrow 0\). Using Taylor expansion for cosecant, we have
\[
G_\ell(u)
=
\left(\frac{1}{\ell^2u^2}+\frac13+O(u^2)\right)
-
\left(\frac{1}{\ell^2u^2}+\frac{1}{3\ell^2}+O(u^2)\right)
=
\frac{\ell^2-1}{3\ell^2}+O(u^2).
\]
Hence
\[
\sup_{u>0}F_\ell(u)
=
\lim_{u\downarrow 0}F_\ell(u)
=
\frac{3\ell^2}{\ell^2-1}.
\]
\end{proof}

\end{document}